\documentclass[12pt, leqno]{article}

\usepackage{float}

\usepackage{amsmath}
\usepackage{amssymb}
\usepackage{amsfonts}
\usepackage{amsthm}
\usepackage{amssymb}
\usepackage{amsbsy}
\usepackage{color}
\usepackage{graphicx}
\usepackage{enumerate}
\usepackage{subfigure}
\usepackage[utf8]{inputenc}

\usepackage[greek,english]{babel}
\usepackage{array}
\usepackage{setspace}
\usepackage{wrapfig}
\usepackage{algpseudocode}
\usepackage{lipsum}

\usepackage{mleftright}

\usepackage{amsmath,amsthm,amssymb}
\usepackage{mathtools}
\usepackage{dsfont}

\usepackage{xcolor}



\usepackage[ruled,linesnumbered]{algorithm2e}

\usepackage[bookmarks]{hyperref}

\newtheorem{theorem}{Theorem}[section]
\newtheorem{corollary}[theorem]{Corollary}
\newtheorem{lemma}[theorem]{Lemma}
\newtheorem{proposition}[theorem]{Proposition}
\newtheorem{assumption}[theorem]{Assumption}

\newtheorem{remark}[theorem]{Remark}

\numberwithin{equation}{section}

\newcommand{\E}{\mathbb{E}}

\newcommand{\N}{\mathbb{N}}
\renewcommand{\P}{\mathbb{P}}

\newcommand{\R}{\mathbb{R}}

\newcommand{\eps}{\varepsilon}

\newcommand{\me}{\mathrm{e}}

\newcommand{\norm}[1]{\left\Vert #1 \right\Vert}

\renewcommand{\d}{{\rm d}}

\usepackage{xcolor}

\definecolor{orange}{rgb}{1,0.3,0.2}

\newcommand{\kfrak}{\mathsf{k}}

\newcommand{\keff}{\ensuremath{\kfrak_{\mathtt{eff}}}}

\usepackage{asymptote}

\newcommand{\hide}[1]{}

\newcommand{\expectation}{\mathbb{E}}
\newcommand{\expectationOperator}[1]{\expectation \mleft[ #1 \mright]}

\newcommand{\markovlaw}[1]{ \mathbb{P}_{#1}}
\newcommand{\markovExpectation}[1]{ \mathbb{E}_{#1} }

\newcommand{\unnormalised}[1]{\gamma_{#1}}
\newcommand{\unnormalisedIntegral}[2]{\unnormalised{#1} \mleft( #2 \mright) }

\newcommand{\normalised}[1]{\eta_{#1}}
\newcommand{\normalisedIntegral}[2]{ \normalised{#1} \mleft( #2 \mright) }

\newcommand{\particleindex}[2]{#1^{#2}}

\newcommand{\particlemeasure}[2]{\particleindex{\eta_{#2}}{#1} }
\newcommand{\particleintegral}[3]{\particlemeasure{#1}{#2} \mleft( #3 \mright) }

\newcommand{\naturals}{\mathbb{N}}

\newcommand{\bbfun}[1]{\mathcal{B}_b \mleft( #1 \mright)}
\newcommand{\probmeas}[1]{\mathcal{P} \mleft( #1 \mright)}
\newcommand{\unkilledprobmeas}[1]{ \hat{\mathcal{P}} \mleft( #1 \mright) }

\newcommand{\supnorm}[2]{ \mleft| \middle| #1 \middle| \mright|_{#2}}

\title{Monte Carlo Power Iteration Algorithms: Convergence, Analysis and Application to Reactor Criticality Calculations}

\author{A. M.~G.~Cox\thanks{Department of Mathematical
		Sciences, University of Bath, Bath, UK., BA2 7AY.
		e-mail: \texttt{a.m.g.cox@bath.ac.uk}, \texttt{td561@bath.ac.uk}, \texttt{jh4272@bath.ac.uk}
	},\,
	T. Davis$^*$,\, J. Hofstadler$^*$ and 
	A. E. Kyprianou\thanks{Department of Statistics, University of Warwick, Coventry, UK. CV4 7AL. 
		e-mail:\texttt{andreas.kyprianou@warwick.ac.uk}}}

\date{\today}

\begin{document}
	\maketitle
	
	\begin{abstract}
		We study Monte Carlo power iteration algorithms in the context of reactor criticality simulations. Under mild conditions on the underlying physical system and the algorithmic structure, we derive explicit error bounds, which are to the best of our knowledge the first rigorous results for such methods which are ubiquitous in nuclear engineering. This is done by exploiting connections to Sequential Monte Carlo algorithms, and extending recent developments about stochastic neutron transport. We demonstrate the applicability of our results by considering a concrete version of the Monte Carlo power iteration algorithm as well as variance reduction schemes. 
	\end{abstract}


	
	\section{Introduction}
	A fundamental task in nuclear engineering is to compute the $\keff$-eigenvalue. Roughly speaking, $\keff \in (0, \infty)$ describes the neutron gain/loss between fission generations, with $\keff < 1$ corresponding to a decreasing $\keff=1$ to a stable and $\keff >1 $ to a growing number of particles. These settings are called \textit{sub-critical} ($\keff <1$), \textit{critical} ($\keff=1$) and \textit{super-critical} ($\keff >1$), respectively.
	
	Neutron transport is modelled via the Neutron Transport Equation (NTE), see \cite{SNTE, SNTE-II, MultiNTE}. Let $D\subseteq\mathbb{R}^3$ be a non-empty, smooth, open, bounded convex domain such that $\partial D$ has zero Lebesgue measure, and $V = \{\upsilon\in \mathbb{R}^3:\upsilon_{\mathtt{min}}\leq  |\upsilon|\leq \upsilon_{\mathtt{max}}\}$. Then, the NTE is given by 
	\begin{equation}\label{equ:NTE}
		\frac{\partial \psi_t[g]}{\partial t}= ({\mathcal T} + \mathcal{S} +{\mathcal F}) \psi_t[g],
	\end{equation}
	where, for sufficiently regular $f \colon D \times V \to \R$, we define
	\begin{equation}
		\label{equ:operators_for_NTE}
		\left\{
		\begin{array}{rl}
			{\mathcal T}{f}(r, \upsilon)  &:=  \upsilon\cdot\nabla{f}(r, \upsilon) - \sigma(r, \upsilon){f}(r, \upsilon)\\
			\\
			{\mathcal S}{f}(r, \upsilon) &:= \sigma_{\texttt{s}}(r, \upsilon)\int_{V}{f}(r, \upsilon') \pi_{\texttt{s}}(r, \upsilon, \upsilon') \d\upsilon' \\
			\\
			{\mathcal F}{f}(r, \upsilon) &: =  \sigma_{\texttt{f}}(r, \upsilon) \int_{V}{f}(r, \upsilon') \pi_{\texttt{f}}(r, \upsilon, \upsilon')\d\upsilon'.
		\end{array}
		\right.
	\end{equation}
	
	Let us explain the terms occurring in \eqref{equ:NTE} and \eqref{equ:operators_for_NTE}. The functions $\sigma_{\mathtt{s}}(r, \upsilon')$ and $\sigma_{\mathtt{f}}(r, \upsilon')$ provide the rates at which scattering and fission\footnote{Throughout all sections we interpret absorption as a fission event producing zero offspring.} happens. We set $\sigma=\sigma_{\mathtt{f}}+ \sigma_{\mathtt{s}}$. If a scattering happened at position $r$ with incoming velocity $\upsilon'$, then $\upsilon\mapsto \pi_{\mathtt{s}}(r, \upsilon', \upsilon)$ is the distribution of the (new) outgoing velocity. Finally, the average fission production at $r$ with incoming velocity $\upsilon$ is $\upsilon\mapsto \pi_{\mathtt{f}}(r, \upsilon', \upsilon)$. The solution to the NTE, i.e. $\psi_t[g](r, \upsilon)$, is called the neutron density. It is time dependent and has to satisfy the initial condition $\psi_0[g](r, \upsilon) = g(r, \upsilon)$.

	Using \eqref{equ:operators_for_NTE}, the $\keff$-eigenvalue problem is as follows (see \cite[Section 1.5]{LM} or \cite[Chapter 7]{nukebook} for details): Find $\keff>0$ and $\omega\colon D \times V \to\R$ such that 
	\begin{equation}\label{equ:keff_problem}
		\left({\mathcal T} + \mathcal{S} + \frac{1}{\keff}{\mathcal F}\right)\omega = 0.
	\end{equation}
	The number $\keff$ describes the criticality of the system. However, it is typically not possible to solve \eqref{equ:keff_problem} exactly. In nuclear engineering, \textit{Monte Carlo power iteration} algorithms provide estimators for $\keff$, see \cite{lux_monte_1991,BlomquistGelbard2002,booth2006power, BrownEtAl2013}, and are the basis of popular tools, like Tripoli \cite{hugot2024overview}, SCONE \cite{raffuzzi2026status}, MCNP \cite[Section 2.8.1.4]{TechReport_2026_LANL_LA-UR-26-23119Rev.1_KuleszaAdamsEtAl}, or OpenMC \cite{romano2015openmc}.
	
	The Monte Carlo power iteration method exploits the fact that $\mathcal{T}, \mathcal{S}$ and $\mathcal{F}$, have a probabilistic representation. By using the \textit{cross sections} $\sigma_{\mathtt{s}}, \sigma_{\mathtt{f}}, \pi_{\mathtt{s}}$ and $\pi_{\mathtt{f}}$ it is possible to model neutrons as stochastic process, see \cite{SNTE,SNTE-II,SNTE-III} and the book \cite{nukebook}. To estimate $\keff$ one mimics \eqref{equ:keff_problem} via a fixed point iteration\footnote{We refer e.g. to \cite[Section 2.1]{pan2023improved} for an explanation regarding the fixed point argument.}: We start with $N \in \N$ neutrons, stored in $\mathcal{N}_0$. Then, starting from $t=1$, we repeat the following:
	
	\medskip
	\begin{itemize}
		\item[1.] Evolve particles in $\mathcal{N}_{t-1}$ until they either undergo fission or leave the system.
		\item[2.] Collect all the offspring resulting from step one in $\mathcal{N}_t$.
	\end{itemize}
	
	\medskip
	A proxy for $\keff$ is the ratio $\mathcal{N}_t/\mathcal{N}_{t-1}$, however, if the system is not critical, the number of neutrons in $(\mathcal{N}_t)_{t \in \N_0}$ either decreases exponentially (for $\keff <1$) or increases exponentially (for $\keff>1$). To overcome this instability issue, an additional \textit{population control} step is used. That is, after evolving particles from $\mathcal{N}_{t-1}$ exactly $N$ offspring are selected to form $\mathcal{N}_t$ (allowing for iid copies). Then, we estimate $\keff$ by  
	\[
	\hat{\kfrak}_t = \frac{\text{number of offspring resulting from }\mathcal{N}_t}{N}.
	\]
	
	Convergence of Monte Carlo power iteration was discussed e.g. in \cite{brissenden1986biases, lux_monte_1991} see also \cite{nowak2016monte,fauvel2025convergence,fauvel2026convergence} and the references therein for more recent developments. These works (among others) provide a practical understanding of Monte Carlo power iteration. However, a rigorous mathematical/theoretical error analysis, with precisely formulated assumptions and error bounds, is to the best of our knowledge still missing from the literature. The present paper closes this gap: Under suitable assumptions, with $\hat{\kfrak}$ based on arithmetic or geometric means of $\hat{\kfrak}_t$, we prove in Section~\ref{sec:PWR_Iteration}: 
	\begin{equation}\label{equ:L_p_error}
		\E \left[\left\vert \hat{\kfrak} - \keff   \right\vert^p  \right]^{1/p} \leq C_p\left(  \frac{1}{\sqrt{N}} + \frac{\tau^T}{k} + (T+k+1)^{1/p}e^{-cN/p}   \right) 
	\end{equation}
	for arbitrary $p \in [1, \infty)$, with constants $C_p,c >0$ and $\tau<1$, where $N \in \N$ is the population size, $T\in \N$ is a burn in period\footnote{That is, only the iterations $T+1 , \dots, T+k$ are used within $\hat{\kfrak}$.}, and $T+k$ is the total number of iterations. Let us briefly discuss this bound:  
	\begin{enumerate}
		\item[1)] The first part, $N^{-1/2}$, is the typical Monte Carlo rate using $N$ particles. 
		\item[2)] The second part, $\tau^T/k$, gets small if we have a large burn in period $T$ (bringing the system closer to stationarity); or if we use many generations $k$ in $\hat{\kfrak}$.
		\item[3)] The third part, $(T+k+1)^{1/p}e^{-cN/p}$ arises by the fact that at every iteration all particles might die (with probability $e^{-cN}$) leading to a degenerate $\hat{\kfrak}_t$.  
	\end{enumerate}
	We also prove a version of \eqref{equ:L_p_error} where only `surviving genealogies' are considered. In this case the estimate simplifies to $\widetilde{C}_p(\frac{1}{\sqrt{N}} + \frac{\tau^T}{k})$. 
	
	Let us emphasise that the assumptions required to deduce \eqref{equ:L_p_error} are met for a large class of applications. Extending the results of \cite{SNTE-III} we give sufficient conditions in terms of the cross sections appearing in the NTE \eqref{equ:NTE}. This allows us to formulate simple conditions directly for the physical system (reactor), see Section \ref{subsec:Physical_assumptions}. The requirements on the generation of $(\mathcal{N}_t)_{t \in \N_0}$ are not restrictive either and are met by a standard variant of the power iteration algorithm (cf. Section \ref{subsec:explicit_fission}) as well as different variance reduction techniques (cf. Section \ref{subsec:variance_reduction}). 
	
	To prove \eqref{equ:L_p_error} we exploit the fact that a Monte Carlo power iteration algorithm can be regarded as Sequential Monte Carlo (SMC) method. This builds upon the fact that the underlying stochastic process is described via an operator $Q$, given by 
	\[
	Q[f] (x) = G(x ) M[f](x ),
	\]
	for bounded functions $f \colon D \times V \to \R$, where $G \colon D \times V \to [0, \infty)$ is called \textit{potential function}, $M$ is a Markov operator and $x \in D \times V$. Here $G(x)$ is expected number of offspring for a fission event happening in configuration $x \in D \times V$, and $M$ describes the behaviour of a neutron between fission events. In this setting the $\keff$ problem is related to a eigenmeasure $\eta$ and eigenfunction $\omega$ such that 
	\[
	\eta Q = \keff\eta \qquad \text{and}\qquad Q \omega = \keff \omega,
	\]
	and $\keff$ is the leading eigenvalue satisfying $\eta(G) = \keff$, cf. Proposition~\ref{prop:lambda-as-potential-integral}. The simulation of (population controlled) particles $(\mathcal{N}_t)_{t \in \N_0}$ can then be described via $Q$ and fits naturally into the SMC framework. Hence, to establish \eqref{equ:L_p_error}, we first show general SMC bounds in Section \ref{sec:L_p_error_bounds} and consider power iteration algorithms in Section~\ref{sec:PWR_Iteration}.
	
	Let us compare our results to the existing literature. As alluded to above, the power iteration method is a standard algorithm in the nuclear engineering literature. However, a rigorous mathematical analysis, including error bounds such as in \eqref{equ:L_p_error}, is to the best of our knowledge not available. The work \cite{MCNTE} investigates basic Monte Carlo methods in the context of criticality calculations. Here \textit{basic} refers to algorithms simulating independent copies of whole neutron paths, rather than population controlled approaches. Let us mention that also deterministic methods can be considered in the context of the $\keff$ problem. The interested reader may consult e.g. \cite{SG} and the references therein.

	Finally, let us comment on how our results fit into the existing SMC literature. There is significant literature dealing with stability of SMC or particle methods, cf. for example the books \cite{del_moral_feynman-kac_2004,ChopinOmiros}. Our general SMC results of Section~\ref{sec:L_p_error_bounds} serve as tool for establishing \eqref{equ:L_p_error}. Thus, emphasis is on a formulation that allows for easily verifiable conditions in terms of the physical system (such as the description in Section~\ref{subsec:Physical_assumptions}), as well as the population control step (such as in Algorithm~\ref{alg:genericPI}).  
	
	Eigenvalue problems for operators $Q$, as defined above, have been studied in \cite{whiteley2017calculating} where an algorithm to approximate the eigenquantities and also error bounds are given. There it is required that $G$ is uniformly bounded away from $0$, which is typically not the case in the NTE setting. The general eigenvalue problem (in the SMC context) was studied e.g. in~\cite{DelMoralMiclo2003Lyapunov,KontoyiannisMeyn2003Spectral,Whiteley2013Stability}. In contrast to those works, we formulate our assumptions more directly in terms of the NTE; in particular, our results allow for vanishing $G$. The work \cite{CerouDelMoralGuyader2011Nonasymptotic} establishes stability results under appropriate conditions on $G$ and the mixing of $M$. Stability of particle systems under drift and small-set conditions were also studied in \cite{Whiteley2012LinearVariance}. Finally, we note that in the paper \cite{DelMoralDoucet2004AbsorbingMedium} error bounds similar to those in \eqref{equ:L_p_error} (stated in SMC/particle system notation) are presented. The setting there (assumptions and results) are somewhat similar, however, our technique allows for more control over the constant. For instance, \cite[Theorem 3]{DelMoralDoucet2004AbsorbingMedium} would lead to a modified upper bound in \eqref{equ:L_p_error} of the form $C_{m}/\sqrt{N}$, with $m = T+k$ being the total number of steps in the algorithm. 
	
	An outline of the paper is as follows. In Section \ref{sec:Stochastic_NTE} we discuss the stochastic model for the NTE; and then formulate the eigenvalue problem in terms of $Q$ in Section~\ref{sec:math_setting}. Background on SMC algorithms and definitions of $\keff$ estimators are in Section \ref{sec:smc}. We prove general SMC $L^p$ error bounds in Section \ref{sec:L_p_error_bounds}  and apply those to power iteration algorithms in Section~\ref{sec:PWR_Iteration}. 
		\section{Stochastic Neutron Transport}\label{sec:Stochastic_NTE}

		To describe how Monte Carlo methods are related to the $\keff$-problem we link the NTE \eqref{equ:NTE} to probability theory.  
		In this section we review the probabilistic model for \eqref{equ:keff_problem}. Additionally, we state assumptions on the underlying system which incorporate physical quantities, e.g. the geometry of a reactor, into the probabilistic framework. 
		
		In this and all later sections, we assume that all random variables, if not explicitly stated otherwise, are defined on a common probability space $(\Omega, \mathcal{F}, \mathbb{P})$ which is sufficiently regular. 
		
		\subsection{The probabilistic model}\label{sec:probabilistic_model}
		
		Note that \eqref{equ:NTE} is understood on some appropriate $L^2(D\times V)$ space. In addition 
		to the initial conditions there are also boundary conditions. This can be used to model the setting that neutron leaving a reactor core is considered lost to the system.

		The basic elements for determining the NTE \eqref{equ:NTE} are the rates $\sigma_{\mathtt{s}}$, $\sigma_{\mathtt{f}}$ and the densities $\pi_{\mathtt{s}}$, $\pi_{\mathtt f}$. 
		It is possible to use these objects to model neutron transport from a
		probabilistic perspective, see \cite{SNTE, SNTE-II, SNTE-III, nukebook}. 
		Namely, $\sigma_{\mathtt{s}}$, $\sigma_{\mathtt{f}}$, $\pi_{\mathtt{s}}$, $\pi_{\mathtt f}$ induce a (not uniquely defined) branching Markov process describing the neutrons.
		More precisely, this branching Markov process is described by finite counting measures on $D\times V$, say 
		\[
		{Y}_t(\cdot) = \sum_{i = 1}^{N_t} \delta_{(r_i(t), \upsilon_i(t))}(\cdot), \qquad t\geq0,
		\]
		where $N_t$ is the number of alive particles at time $t$, and $((r_i(t), \upsilon_i(t)), i = 1, \cdots, N_t)$ represents their configuration.
		
		For a finite counting measure $\mu$ on $D \times V$, say $\mu$, we write $\mathbb{P}_\mu$ for the underlying probability measure of $(Y_t)_{t \geq 0}$. For measurable and bounded $f: D\times V\to [0,\infty)$, we define
		${Y}_t[f] = \sum_{i = 1}^{N_t}f(r_i(t), \upsilon_i(t))$. The mean
		semigroup is defined as
		$\Psi_t[f](r,\upsilon) =\mathbb{E}_{\delta_{(r,\upsilon)}}[{Y}_t[f]]$,
		for $r\in D, \upsilon\in V$ and $t\geq0$, with $\Psi_0[f] = f$. We have
		\[
		\langle f, \psi_t[g]\rangle =  \langle \Psi_t[f], g\rangle , \qquad t\geq0,
		\]
		where $ \langle\cdot,\cdot\rangle $ is the inner product on $L^2(D\times V)$, i.e. space on which \eqref{equ:NTE} is defined. 
		Furthermore, the  semigroup $(\Psi_t, t\geq0)$ 
		solves 
		\begin{equation*}
			\frac{\partial \Psi_t[f]}{\partial t}= ({\mathcal T} + \mathcal{S} +{\mathcal F}) \Psi_t[f],
		\end{equation*}
		with ${\mathcal T}$, $\mathcal{S}$ and ${\mathcal F}$ defined as in \eqref{equ:operators_for_NTE}. We refer to the works \cite{SNTE, MultiNTE} for more details about stochastic neutron transport. 
			
			Although $(Y_t)_{t \geq 0}$ and $(\Psi_t)_{t \geq 0}$ directly connect the NTE to probability theory we cannot use these objects directly within the $\keff$-problem. However, there is a closely related process that can be used instead.
			Namely, think of each line of descent in $(Y_t)_{t \geq 0}$ as a genealogy. We define the process $({Z}_n)_{ n \in \N_0}$, for $n\geq 0$, by 
			\[
			{Z}_n(\cdot) = \sum_{i = 1}^{N^{(n)}}\delta_{(r^{(n)}_i, \upsilon^{(n)}_i )}(\cdot),
			\]
			where $N^{(n)} \in \N_0$ is the number of particles in the $n$-th generation of their genealogical lines of descent, and $\{(r^{(n)}_i, \upsilon^{(n)}_i ), i = 1, \cdots, N^{(n)}\}$ are the corresponding position-velocity configurations. Here `$n$-th generation' is the collection of neutrons produced at the $n$-th fission event in each genealogical line of descent form the initial ancestor at time zero. We follow \cite{nukebook} and call $(Z_n)_{n \in \N_0}$ the \textit{neutron generational process} (NGP).

			For measurable and bounded $g \colon D \times V \to \R$, set ${Z}_n[g] = \sum_{i = 1}^{N^{(n)}}g(r^{(n)}_i, \upsilon^{(n)}_i )$. 
			Furthermore, we define the expectation semigroup corresponding to $(Z_n)_{ n \in \N_0}$ by 
			\begin{equation}
				{Q}_n[g](r, \upsilon) = \mathbb{E}_{\delta_{(r, \upsilon)}}\left[{Z}_n[g] \right], \qquad n\geq 0, r\in D, \upsilon\in V,
				\label{expsemi}
			\end{equation}
			with ${Q}_0[g]: = g$. We note that indeed ${Q}_{n+m}= {Q}_n({Q}_m)$ for $n,m\in\mathbb{N}_0$. 
			
			As is shown below, the semigroup $(Q_n)_{n \in \N_0}$ allows to connect the $\keff$-problem to the theory of SMC methods. However, before explaining this, we provide assumptions on the physical system (reactor), i.e.~conditions on the cross sections $\sigma_{\mathtt{f}}, \sigma_{\mathtt{s}},\pi_{\mathtt f},\pi_{\mathtt{s}}$.

			\subsection{Assumptions on the physical model}\label{subsec:Physical_assumptions}
			
			In this section we discuss assumptions about the generic nuclear reactor we impose. 
			Our conditions are sufficiently mild as to be relevant for most if not all systems which are modelled in practical situations.
			
			Recall, that the domain where neutrons evolve is $D$, while $V \subseteq \R^3$ describes their velocity. From this point we simply call $(r,v) \in D\times V$ a \textit{configuration} of a neutron. 
			
			We assume that $D \subseteq \R^3$ is a bounded, open and convex set with $\lambda(\partial D)=0$. The boundary $\partial D$ corresponds to the boundary of the reactor, and particles reaching this boundary will be considered as `having left the system'. 
			
			The fissile material is modelled as follows. There exists a subset
			$D_f$, referred to as the \textit{fissile region}. For example, this
			might correspond to the uranium used within a reactor. Additionally,
			there is a \textit{non-fissile} region $D_0 = D \setminus D_f$.
			
			\smallskip
			\begin{description}
				\item[Assumption F.] Neutrons can undergo fissions only if they are in $D_f$. Conversely, in $D_0$ no fission is possible. We assume that $D_f$ is surrounded by $D_0$, i.e. 
				\[
				\text{dist}(D_f, \partial D) >0.
				\]
				Furthermore, we assume that there exists an open ball $B\subseteq D$ of positive radius such that 
				\[
				\inf_{(r,v,v') \in B \times V \times V} \sigma_{\texttt{f}}(r,v)\pi_{\texttt{f}}(r,v,v') > 0.
				\]
			\end{description}
			
			\smallskip
			
			Next we state our conditions regarding scattering events: 
			
			\smallskip
			
			\begin{description}
				\item[Assumption S.]  The probability of scattering is
				uniformly bounded away from $0$ for any configuration. That is, 
				$\inf_{(r,v) \in D \times V}\sigma_{\mathtt{s}}(r, v)>0$.
			\end{description}
			
			\smallskip
			
			Our second condition addresses neutrons that have just undergone
			scattering. Recall that this is modelled via the density
			$\pi_s(r,v, v')$ (where the neutron is at position $(r,v)$ when the
			scatter occurs and $v'$ is the outgoing direction). Hence in total we
			obtain
			\[
			\inf_{(r,v,v') \in D \times V\times V} \sigma_{\mathtt{s}}(r,v)\pi_s(r,v, v') > 0.
			\]
			Finally, we assume that all rates and particle production are bounded
			above:
			
			\smallskip
			
			\begin{description}
				\item[Assumption B] All cross-sections
				$\sigma_{\mathtt{f}}, \sigma_{\mathtt{s}}, \pi_{\mathtt f},
				\pi_{\texttt{s}}$ and the maximum number of particles that can be released
				at single fission events are uniformly upper bounded.
				
			\end{description}

			\section{Mathematical setting}\label{sec:math_setting}
			In this section, we develop in more (mathematical) detail the tools
			that we require to formulate abstract power iteration algorithms. Let
			$(E, \mathcal{E})$ be a measurable space and $\bbfun{E}$ the set of
			all measurable and bounded $f \colon E \to \R$.  We write
			$\supnorm{\cdot}{E}$ for the supremum norm on $E$.  Typically,
			$E=(D,V)$, however, for the sake of a general theory we keep the space
			$(E, \mathcal{E})$ abstract.
			
			\subsection{The abstract $\keff$ setting}
			
			We slightly overload notation and consider a non-negative and bounded linear operator ${Q} : \bbfun{E} \rightarrow \bbfun{E}$, such that $f\mapsto {Q}[f] \in \bbfun{E}$, for $f\in\bbfun{E}$. 
			In the context of nuclear applications, where $E = (D, V)$, $Q$ corresponds to the expectation semigroup of the NGP $(Z_n)_{n \in \N_0}$. That is, $Q= Q_1$ with $Q_1$ as in Section~\ref{sec:probabilistic_model}.
			
			Note that $(x, A) \mapsto Q[\mathds{1}_A] (x) = : Q(x,A)$ is a generalised Markov transition kernel. Indeed, for $A \in \mathcal{E}$ the mapping $x \mapsto Q(x, A)$ is measurable, however, $Q(x,E) \neq 1$ in general.    
			Let $\probmeas{E}$ be the set of probability measures on $(E, \mathcal{E})$. For $\mu \in \probmeas{E}$, we define
			\[
			\mu{Q}(A) = \int_E {Q}(x, A)\mu(\d x), \qquad A\in\mathcal{B}.
			\]
			
			We assume that ${Q}$ satisfies the following Perron-Frobenius type condition:
			\begin{assumption} \label{assumption:Q-perron-frobenius}
				There exist $\keff \in (0, \infty)$, $\omega \in \bbfun{E}$, and $\eta \in \probmeas{E}$ such that:
				\begin{itemize}
					\item[(i)] ${Q}[\omega] = \keff \omega$ pointwise on $E$;
					\item[(ii)] $\eta {Q}(A) = \keff \eta(A)$ for each $A \in \mathcal{E}$;
					\item[(iii)] If $\alpha \in \mathbb{C}$ is any other eigenvalue of ${Q}$, then $\left| \alpha \right| < \keff$. 
				\end{itemize}
			\end{assumption}
			
			We call $(\keff, \omega, \eta)$ in Assumption \ref{assumption:Q-perron-frobenius} the leading eigentriple of ${Q}$. Note, that here we again overload the notation, since for general linear operators $Q$ the eigenvalue $\keff$ in Assumption \ref{assumption:Q-perron-frobenius} may not have a physical interpretation. 
			In the nuclear context, where $Q$ corresponds to the
			semigroup of the NGB, Assumption \ref{assumption:Q-perron-frobenius}
			exactly refers to the physical quantities which are considered in
			\eqref{equ:keff_problem},  see \cite{SNTE-III} or \cite[Chapter
			7]{nukebook}. 
			
			In order to fit well into the SMC literature, we assume that $Q$ is the product of a potential function $G \colon E \to [0, \infty]$ and a Markov kernel $M \colon E \times \mathcal{E} \to [0,1]$. 
			To be more precise, we assume that the potential function $G $ is measurable and bounded, and that 
			\[
			Q(x, A) = G(x) M(x, A),
			\]
			for any $x \in E$ and $A \in \mathcal{E}$.
			Note that alternatively we can write $G(x) = Q(x, E)$.
			Using this together with Assumption \ref{assumption:Q-perron-frobenius} we immediately obtain: 
			
			\begin{proposition} \label{prop:lambda-as-potential-integral}
				If Assumption \ref{assumption:Q-perron-frobenius} is true, then $\eta(G) = \keff$.
			\end{proposition}

			Recall that in the context of the NTE, $G(x)$ stands for the average number of offspring a neutron produces if it is about to undergo fission at configuration $x = (r,v) \in D \times V$.
			To incorporate the fact that there are non-fissile regions we have to work in the setting where $E_0 = \{x \in E \colon G(x) =0\}$ is non-empty.
			To exclude the situation that any un-absorbed particle is always killed within one step we assume the following: 
			\begin{assumption} \label{assumption:absorption}
				For each $x \in E \setminus E_0$ we have $M (x, E\setminus E_0 ) > 0$. 
			\end{assumption}
			This assumption is standard in the interacting particle system literature, compare e.g.~with \cite[Equation (2.16) on p. 67]{del_moral_feynman-kac_2004}. In the setting of Section \ref{subsec:Physical_assumptions} this assumption is satisfied due to the uniform lower bounds on the cross-sections. 
			
			\subsection{Interacting particle systems notation}

			In this section we introduce the required notation and some standard facts from the interacting particle systems literature. For further details and proofs we refer e.g. to \cite{del_moral_feynman-kac_2004}.

			Fix some $\eta_0 \in \probmeas{E}$ such that $\eta_0 (E\setminus E_0) =1$. Let $X = ( X_n)_{n\in\N_0}$ be a time-homogeneous Markov chain on $(E, \mathcal{E})$ with transition kernel $M$. Let $\markovlaw{\eta_0}$ be the law of this process when initialised from $X_0 \sim \eta_0$ and $\markovExpectation{\eta_0}$ be the corresponding expectation.

			For each $n \in \naturals$ we define $\unnormalised{n}$ and $\normalised{n}$, respectively a finite measure and a probability measure on $(E, \mathcal{E})$, as follows: 
			\begin{align} \label{eq:normalised-def}
				\unnormalisedIntegral{n}{f}  =  \E_{\eta_0}\left[ f \left( X_n \right) \prod_{k=1}^{n-1} G \left( X_k \right) \right] 
				\qquad \text{and} \qquad
				\normalisedIntegral{n}{f} = \frac{ \unnormalisedIntegral{n}{f} }{ \unnormalisedIntegral{n}{1} },  
			\end{align}
			for $f \in \bbfun{E}$. By Assumption \ref{assumption:absorption}, $\unnormalisedIntegral{n}{1} > 0$ for every $n \in \naturals$. 

			Furthermore, we have
			\begin{equation} \label{eq:unnormalised-product}
				\unnormalisedIntegral{n}{f} = \normalisedIntegral{n}{f} \prod_{k=0}^{n-1} \normalisedIntegral{k}{G}.
			\end{equation}
			Setting $Q^0(f) = f$, and $Q^n = Q (Q^{n-1})$ for $n \in \N$, then 
			\begin{equation*}
				\unnormalisedIntegral{n}{f} = \eta_0 Q^n(f) \qquad \text{and} \qquad \unnormalisedIntegral{n+1}{f} = \unnormalisedIntegral{n}{Q(f)},
			\end{equation*}
			for $n \in \naturals$ and $f \in \bbfun{E}$. 
			
			It is well understood that the flow of probability measures $( \eta_n)_{n \in \N}$ has a similar dynamical structure: Define  $\unkilledprobmeas{E}= \{ \mu \in \probmeas{E} : \mu(E_0) < 1 \}$, so that $\mu(G) > 0$ for each $\mu \in \unkilledprobmeas{E}$. Let the non-linear operator $\Phi : \unkilledprobmeas{E} \rightarrow \unkilledprobmeas{E}$ be defined by
			\begin{equation*}
				\Phi(\mu)(f) =  \frac{ \mu Q(f) }{ \mu(G) }, \qquad \qquad \left( \mu, f \right) \in \unkilledprobmeas{E} \times \bbfun{E}.
			\end{equation*}
			By Proposition 2.5.2 of \cite{del_moral_feynman-kac_2004}, for each $n \in \N_0$, 
			\begin{equation*}\label{etarecursion}
				\eta_{n+1} = \Phi(\eta_n) \qquad \text{and} \qquad \eta_n = \Phi^n (\eta_0), 
			\end{equation*} 
			for $n \in \N_0$, where $\Phi^0(\eta) = \eta$ and $\Phi^n(\eta) = \Phi (\Phi^{n-1}(\eta))$.

			In order to prove our results, we make use of the following assumption. 
			
			\begin{assumption}\label{ass_perron_frobenius_assumption}
				Additionally to Assumption \ref{assumption:Q-perron-frobenius}, there exist $C \in (0, \infty)$ and $\tau \in [0,1)$ such that for $x \in E \setminus E_0$ and measurable $g \colon E \to \R$ with $\norm{g}_E \leq 1$, we have 
				\[
				\left\vert \frac{Q^n[g](x)}{\keff^n} - \omega(x)\eta(g) \right \vert \leq C \tau^n
				\]
				for any $n\in \N$. Furthermore, we have $0 < \inf_{x \in E\setminus E_0} \omega(x) \leq \sup_{x \in E \setminus E_0} \omega(x) < \infty$. 
			\end{assumption}
			
			Let us comment on this assumption. Existence of the eigen-quantities $(\keff, \omega, \eta)$ and exponential convergence as in Assumption \ref{ass_perron_frobenius_assumption} can be deduced from \cite{CV} or \cite{champagnat2023general}. To apply these results it is required to show that the Neutron Branching Process $Y$, see Section \ref{sec:probabilistic_model}, returns with sufficiently high probability to certain regions conditioned on survival. Finally, the `largest eigenvalue' property is a consequence of the exponential convergence (as stated in Assumption~\ref{ass_perron_frobenius_assumption} but only assuming existence of the eigen-quantities rather than Assumption \ref{assumption:Q-perron-frobenius} directly). 
			
			The results of \cite{CV} have been used in \cite{SNTE-III} to verify Assumption \ref{ass_perron_frobenius_assumption}. However, the physical systems considered in \cite{SNTE-III} allow for fission to happen everywhere in the reactor. This is not possible in the setting of the present paper. Yet, by slightly adapting the proof it is still possible to proceed as in \cite[Section 5]{SNTE-III}. We highlight the main ideas/differences (fitting into the set up of Section \ref{subsec:Physical_assumptions}), but omit technical details:  
			\begin{itemize}
				\item[1.] Choose the `minorising measure' $\nu$ to be the uniform distribution on $B \times V$, rather than the uniform distribution on the whole space $D \times V$.
				\item[2.] Considering paths which have two scatters, the first one pointing towards some ball $B' \subseteq B$ and the second one in $B'$, and using the lower bounds on fission and scattering cross-sections, shows \cite[A1]{CV} with $n_0=1$. 
				\item[3.] Using a similar `seven jump' argument as in \cite[Chapter 5]{SNTE-III} verifies \cite[A2]{CV}.
			\end{itemize}

			We note that from \cite{CV} it does not follow directly that $\omega$ is uniformly bounded away from $0$ on the set $E \setminus E_0$. However, we claim that in the setting of Section \ref{subsec:Physical_assumptions} this is true: To this end, note that $M(x , A) \geq c \nu(A)$ for any $x \in E\setminus E_0$ and measurable $A \subset E$, with $\nu$ being the same minorising measure as above, and where $c>0$ is independent of $x,A$. Using this, the fact that $\omega$ is positive on $E\setminus E_0$ (which indeed does follow from \cite{CV}) and the eigenfunction property we obtain 
			\begin{align*}
				\omega(x) = \frac{1}{\keff} Q\omega (x) = \frac{G(x)}{\keff} \int_E \omega(x') M(x, \d x') \geq \frac{\kappa}{\keff} \int_{E \setminus E_0} \omega(x') M(x, \d x') >0,
			\end{align*}
			for any $x \in E \setminus E_0$, where the lower bound does not depend on $x$. 
			In particular, this argument gives:
			\begin{lemma}
				Under Assumptions S, F, and B of Section~\ref{subsec:Physical_assumptions},
				Assumptions~\ref{assumption:Q-perron-frobenius} and \ref{ass_perron_frobenius_assumption} hold.
			\end{lemma}

			We now give a consequence of these assumptions.
			
			\begin{proposition}\label{prop:exponential_congergence_in_total_variation}
				Let Assumption \ref{ass_perron_frobenius_assumption} be satisfied. Then there exists $\widetilde{C} < \infty$ such that for any $x \in E \setminus E_0$ and any measurable $g \colon E \to \R$ with $\norm{g}_E \leq 1$ we have
				\[
				\left\vert \frac{Q^n [g](x)}{Q^n[\mathds{1}](x)} - \eta(g) \right\vert \leq \widetilde{C} \tau^n,
				\]
				for any $n \in \N$, where $\tau$ is as in Assumption \ref{ass_perron_frobenius_assumption}. 
				Consequently, for any $x \in E \setminus E_0$, we also have $\norm{\Phi^n (\delta_x) - \eta}_{tv} \leq \widetilde{C}\tau^n$ for any $n \in \N$. 
			\end{proposition}
			
			\begin{proof}
				Let $n_0 \in \N$ be large enough such that $C \tau^{n_0} < \omega(x)$ for any $x \in E \setminus E_0$. Set $T_n[g] (x) = \keff^{-n} \omega(x) Q^n[g](x)$.
				
				Multiplying numerator and denominator by $\keff^{-n} \omega(x)$ we obtain 
				\[
				\left\vert\frac{Q^n [g](x)}{Q^n[\mathds{1}](x)} - \eta(g) \right\vert 
				= 
				\left\vert  \frac{T_n[g](x)}{T_n[g](x)} - \eta(g) \right\vert 
				=
				\left\vert\frac{T_n [g](x) - \eta(g) T_n[\mathds{1}](x)}{T_n[\mathds{1}](x) }\right\vert.
				\]
				By Assumption \ref{ass_perron_frobenius_assumption}, we get that the denominator in the above expression is lower bounded by $\inf_{x \in E \setminus E_0}\omega(x) - C\tau^{n_0}>0$ for any $n \geq n_0$, while the numerator decreases as $\tau^n$. 
				Finally, the statement for the total variation distance follows by well known duality formulas. 
			\end{proof}
			
			The following result shows that in the notation of this section there are several different quantities which are converging to $\keff$. 
			These insights can later be used to define different estimators in the Monte Carlo context. 
			
			\begin{proposition} \label{prop:lambda-as-feynman-kac-limit}
				Under Assumption \ref{ass_perron_frobenius_assumption}, for
				$\eta_0 = \delta_x$, with $x \in E \setminus E_0$, we have
				\begin{equation}
					\keff = \lim_{n \rightarrow \infty}\gamma_n(1)^\frac1{n}= \lim_{n \rightarrow \infty} \eta_0 Q_n(1)^\frac1{n} = \lim_{n \rightarrow \infty} \frac1{n} \sum_{k=0}^{n-1} \eta_k (G). 
				\end{equation}
			\end{proposition}
			\begin{proof}
				Assumption \ref{ass_perron_frobenius_assumption}, and the fact that $\frac{\eta_0 Q^n(G )}{\eta_0Q^n(\mathds{1})} = \normalisedIntegral{n}{G} = \Phi^n (\eta_0)(G)$ yield $\lim_{n \rightarrow \infty} \normalisedIntegral{n}{G} = \keff$. Thus, by positivity also the arithmetic and geometric mean converges, 
				\begin{equation*}
					\lim_{n \rightarrow \infty} \frac1{n} \sum_{k=0}^{n-1} \normalisedIntegral{k}{G} = \lim_{n \rightarrow \infty} \left( \prod_{k=0}^{n-1} \normalisedIntegral{k}{G} \right)^\frac1{n} = \keff.
				\end{equation*}
				Finally, we recall from \eqref{eq:unnormalised-product} that $\prod_{k=0}^{n-1} \normalisedIntegral{k}{G} = \unnormalisedIntegral{n}{1} = \eta_0 Q_n(1)$.
			\end{proof}
			
			\section{Sequential Monte Carlo Algorithms}\label{sec:smc}
			Sequential Monte Carlo (SMC) algorithms are a popular tool in applications to generate samples from a given target distribution, we refer to \cite{ChopinOmiros} for an introduction. In this section we review the required ingredients which are used later to analyse Monte Carlo power iteration algorithms from an SMC perspective.

			\subsection{SMC and power iteration for neutron transport}\label{subsec:SMC4NTE}
			
			In this section, we describe the SMC scheme that is of relevance later to analyse Monte Carlo power iteration methods.  We recall the setting of Section~\ref{sec:Stochastic_NTE}: The linear operator $Q$, acting on
			$\bbfun{E}$, has the form
			\[
			Q(x, A) = G(x) M(x, A),
			\]
			for $x \in E$ and $A \in \mathcal{E}$, where $G$ is a (measurable and bounded) potential function and $M$ a Markov kernel.
			Now, we evolve $N \in \N$ weighted particles as follows: The initial particles, say $\xi_1^0, \dots, \xi_N^0$, are sampled iid according to some initial distribution $\mu_0$, where for simplicity we assume $\mu_0(E \setminus E_0)=1$.
			Then, iteratively we repeat the following to compute $\xi_1^{t+1}, \dots, \xi_N^{t+1}$ given $\xi_1^{t}, \dots, \xi_N^{t}$, assuming $G(\xi_j^t) \neq 0$ for at least one $j$: 
			\begin{itemize}
				\item[1.] \textbf{Resampling:} Select\footnote{Here repetition is allowed, and if a particle is selected more than once, this refers to generating iid copies that can evolve independently afterwards. } $N$ particles, say $\widetilde{\xi}_1^{t+1} , \dots , \widetilde{\xi}_N^{t+1}$ from $\xi_1^{t}, \dots, \xi_N^{t}$ 	so that the expected number of copies of $\xi_j^t$ in $\widetilde{\xi}_1^{t}, \dots, \widetilde{\xi}_N^{t}$ is $W_j^t$ given by 	 
				\[
				W_j^{t} = \frac{G(\xi_j^{t})}{\sum_{i = 1}^N G(\xi_i^{t})},
				\]
				\item[2.] \textbf{Evolution:} Compute $\xi_1^{t+1}, \dots, \xi_N^{t+1}$ by sampling independently
				\[
				\xi_j^{t+1} \sim M(\widetilde{\xi}_j^t, \cdot). 
				\]
			\end{itemize}
			
			Let us discuss this generic scheme in the nuclear context. In this setting, $G(x)$, for $x \in E$, corresponds to the average number of offspring. The Markov kernel $M(x, \cdot)$ is related to evolving a particle that was generated in $x \in E$ (through fission) to its next fission event.   
			
			In the SMC literature this scheme is called the \textit{bootstrap particle filter}, cf. \cite{ChopinOmiros} and also \cite{whiteley2017calculating} in the context of eigenvalue estimation. 
			
			\subsection{Estimators for $\keff$}
			In this section we explain how we can estimate $\keff$. The main idea is to use the relations of Proposition \ref{prop:lambda-as-feynman-kac-limit}.
			To this end, we have to estimate for $n \in \N$, respectively, 
			\[
			\gamma_n (\mathds{1}), \qquad \text{and/or} \qquad \eta_n(G),
			\]
			by using the sequence of particles $(\xi_1^t, \dots , \xi_N^t)_{t \in \N}$ from the bootstrap particle filter discussed in the previous subsection. 
			This can be done by 
			\begin{equation}\label{equ:eta_estimator_definition}
				\eta_t^N (\cdot ) = \frac{1}{N} \sum_{j=1}^{N} \delta_{\xi_j^t}(\cdot),
			\end{equation}
			where $\delta_x(\cdot)$ denotes the Dirac delta distribution with mass at $x \in E$. 
			The above discussion suggests that estimating $\keff$, can be done via arithmetic or geometric means, that is  
			\begin{equation}\label{equ:keff_estimator_arithmetic_mean}
				\kfrak_{T,k}^{ari} = \frac{1}{k}\sum_{t=T+1}^{T+k} \eta_t^N (G) \qquad \text{and} \qquad 	\kfrak_{T, k}^{geo} = \left( \prod_{t=T+1}^{T+k} \eta_t^N (G) \right)^\frac1{k}.
			\end{equation}
			
			\section{$L^p$ error bounds}\label{sec:L_p_error_bounds}
			In this section we present our theoretical results about the stability of the estimators $\kfrak_{T,k}^{ari}$ and $\kfrak_{T,k}^{geo}$ of \eqref{equ:keff_estimator_arithmetic_mean}.
			Here and for the rest of this section we fix some initial distribution $\mu_0$ satisfying $\mu_0(E\setminus E_0) =1$ for the initial particles $\xi_1^0, \dots, \xi_N^0$. Implicitly, the probability space $(\Omega, \mathcal{F}, \mathbb{P})$, and hence also $\E$, depend on this. However, we omit this dependence to increase the readability. 
			
			\subsection{Extinction probabilities}
			
			Since we work in a setting where $G(x) =0$ for certain $x \in E$, it is possible that in the whole particle system in the generic particle algorithm of Section \ref{sec:smc} dies out. 
			More precisely, this means that from generation $t$ to $t+1$, after the resampling step, all new particles $\xi_1^{t+1}, \dots , \xi_N^{t+1}$ are in $E_0$, where they are considered as killed. 
			We define the first time that this happens as $\tau(N)$, i.e.,
			
			\begin{equation}\label{equ:definition_extintion_time}
				\tau(N) = \inf\left\{t \in \N \colon (\xi_1^t, \dots , \xi_N^t) \in \left(E_0\right) ^N   \right\} .
			\end{equation}
			Note that equivalently, $\tau(N) =  \inf\left\{t \in \N \colon G(\xi_j^t) = 0 \text{ for all } 1 \leq j \leq N  \right\}$. Hence, 
			\begin{equation}\label{equ:extinction_weights}
				\left\{ \tau(N) \geq n \right\} = \bigcap_{k=0}^{n-1} \left\{ \eta_{t}^{N}({G}) > 0 \right\}.
			\end{equation}
			By \eqref{equ:extinction_weights} it follows that $\tau(N)$ is a predictable stopping time w.r.t. the natural filtration of the particles, denoted by $(\mathcal{F}_t)_{t \in \N_0}$.
			
			For the sake of obtaining good estimates for the probability of surviving, we work with the following `quantitative' version of Assumption \ref{assumption:absorption}. 
			
			\begin{assumption}\label{ass:minimal_survival_probability}
				There exists $\kappa\in (0,1)$ such that for any $x \in E\setminus E_0$ we have
				\[
				M(x, E \setminus E_0) \geq \kappa \qquad\text{and} \qquad G(x) \geq \kappa.
				\]
			\end{assumption}
			
			Note that Assumption \ref{ass:minimal_survival_probability} is true in the setting of Section \ref{subsec:Physical_assumptions} due to the uniform bounds on all involved cross-sections.

			The next result is a version of \cite[Theorem 7.4.1]{del_moral_feynman-kac_2004} in the setting of Assumption~\ref{ass:minimal_survival_probability}. In order to obtain explicit constants we include the short proof. 
			
			\begin{proposition}\label{prop:survial_probability_estimate}
				If Assumption \ref{ass:minimal_survival_probability} is satisfied, then for any $t \in \N$ we have 
				\[
				\mathbb{P} \left[  \tau(N) \leq t \right] \leq t e^{-a N},
				\]
				with $a = - \log(1-\kappa) \in (0, \infty)$, where $\log$ refers to the natural logarithm.
			\end{proposition}
			
			\begin{proof}
				First, observe that by Assumption~\ref{ass:minimal_survival_probability}, $\mathbb{P}[\tau(N) >i] >0 $ for $i \in \N$.
				
				Note that if the system has survived until time $i$, then $\xi_j^i \in E \setminus E_0 $ for at least one index $j \in \{1, \dots, N\}$. In this case, by definition of the Bootstrap particle filter, $\widetilde{\xi}_j^{i+1} \in E \setminus E_0$ for any $j \in \{1, \dots, N\}$. Hence, by Assumption \ref{ass:minimal_survival_probability} we obtain $\mathbb{P} \left[ \xi_j^{i+1} \in E_0 \; \vert \;\tau(N) >i \right] \leq 1-\kappa$. Since particles evolve independently given $\widetilde{\xi}_1^{i}, \dots, \widetilde{\xi}_N^i$, the probability of `dying out from iteration $i$ to $i+1$' is bounded by
				\[
				\mathbb{P} \left[ \tau(N) = i+1 \;\vert \;\tau(N) >i \right] =
				\mathbb{P} \left[ \bigcap_{j=1}^N\{\xi_j^{i+1} \in E_0\} \; \bigg\vert \;\tau(N) >i \right] \leq (1-\kappa)^N.
				\]
				Hence, $\mathbb{P}\left[ \tau(N) = i+1 \right] = 
				\mathbb{P} \left[ \tau(N) = i+1 \, \vert \, \tau(N) > i\right] \cdot \mathbb{P}[\tau(N) >i] \leq (1- \kappa)^N$.
				Finally, 
				\[
				\mathbb{P}[\tau(N) \leq t] = \sum_{i = 1}^t \mathbb{P}[\tau(N) = i] \leq t (1-\kappa)^N.
				\]
			\end{proof}
			
			\subsection{Error bounds}
			
			In this section we present the error bounds for the estimators $\kfrak_{T, k}^{ari}$ and $\kfrak_{T, k}^{geo}$. First, we study how well we can approximate the measures $(\eta_{t})_{t \in \N}$. Then we move to our actual error bounds. 
			
			\subsubsection{Stability of the particle measures $\eta_{t}^N$}

			The estimators $\kfrak_{T, k}^{ari}$ and $\kfrak_{T, k}^{geo}$ both make use of the particle measures $\eta_{t}^N$, which serve as approximation for the flow $(\eta_{t})_{t \in \N}$. 
			It is therefore useful to understand how well we recover this flow in order to prove bounds for estimating $\keff$.

			\begin{assumption} \label{assumption_E_p} Let $p \in [1, \infty)$.
				There exists a constant $C_p \in (0, \infty)$, such that for any
				measurable and bounded $f\colon E \to \R$ and any $t \in \N$ we have
				\[
				\E \left[\mathds{1}_{\{\tau(N) > t\}} \left\vert \eta_{t}^N(f) - \Phi(\eta_{t-1}^N) (f) \right\vert^p \, \bigg\vert \, \mathcal{F}_{t-1}\right]^{1/p} \leq \mathds{1}_{\{\tau(N) > t\}}\frac{C_p \norm{f}}{\sqrt{N}}.
				\]
			\end{assumption}
			
			The main result of this section is the following.
			
			\begin{theorem}\label{thm:stability_particle_approx}
				Suppose Assumptions \ref{ass_perron_frobenius_assumption}, \ref{ass:minimal_survival_probability} and \ref{assumption_E_p} (for some $p \in [1, \infty)$) are true.
				Then, there exists some constant $K_p \in (0, \infty)$ such that for any $t \in \N$ and bounded measurable $f \colon E \to \R$ we have 
				\[
				\E\left[\mathds{1}_{\{\tau(N)> t\}} \left\vert \eta_{t}(f) - \eta_{t}^N(f)  \right\vert^p  \right]^{1/p} \leq \frac{K_p \norm{f}}{\sqrt{N}}. 
				\]  
			\end{theorem}
			
			We split the proof into several parts. 
			The main idea is to use following error decomposition:
			set $f_t = f- \eta_{t}(f)$ and $A_t =\{\tau(N)> t\}$.
			Under the assumptions and using the same notation as in Theorem~\ref{thm:stability_particle_approx} we have 
			\begin{align}\label{equ:error_decomp}
				\E\left[\mathds{1}_{A_t} \left\vert \eta_{t}(f) - \eta_{t}^N(f)  \right\vert^p  \right]^{1/p} 
				&\leq  
				\E\left[\mathds{1}_{A_t} \left\vert \eta_{t}^N(f_t) - \Phi(\eta_{t-1}^N)(f_t) \right\vert^p  \right]^{1/p}  \\
				&\quad +\sum_{k=1}^{t-1} \E\left[ \mathds{1}_{A_t} \left\vert \Phi^{t-k} (\eta_k^N) (f_t) - \Phi^{t-k+1} (\eta_{k-1}^N) (f_t)  \right\vert^p \right]^{1/p} \nonumber\\ 
				&\quad +
				\E\left[ \mathds{1}_{A_t} \left\vert \Phi^t(\eta_0^N)(f_t) \right\vert^p  \right]^{1/p}. \nonumber
			\end{align}
			
			To prove Theorem \ref{thm:stability_particle_approx} we now estimate
			all terms appearing in \eqref{equ:error_decomp}. We do so first show
			several auxiliary results.

			\begin{lemma} \label{lem:bound_semigroup}
				With the assumptions and notations as in Theorem~\ref{thm:stability_particle_approx}, there exists a constant $C\in (0, \infty)$ such that for any $t, k \in \N$ with $k \leq t$ and any bounded and measurable $f$ we have  
				\[
				\vert Q^{t-k} (f_t)(x) \vert \le C\tau^{t-k} \norm{f} Q^{t-k}(\mathds{1})(x),
				\]
				for any $x \in E$. 
			\end{lemma}
			
			\begin{proof}
				Let $x \in E \setminus E_0$ and $P_{t-k}g(x) = \frac{Q^{t-k}g(x)}{Q^{t-k}\mathds{1}(x)}$, for sufficiently regular $g$. Then, 
				\[
				Q^{t-k} (f_t)(x) = P_{t-k}(f_t)(x) Q^{t-k}(\mathds{1}) (x).
				\]
				Note that $P_{t-k}(f_t)(x) = P_{t-k}(f)(x) - \eta_t(f) = \Phi^{t-k}(\delta_x)(f) - \Phi^t(\mu_0)(f)$, such that 
				\[
				\vert P_{t-k}(f_t)(x) \vert \leq \vert \Phi^{t-k}(\delta_x)(f) - \eta(f) \vert + \vert  \Phi^t(\mu_0)(f) - \eta(f) \vert \leq 2\widetilde{C}\norm{f}\tau^{t-k}, 
				\]
				by Proposition \ref{prop:exponential_congergence_in_total_variation}. This implies the statement for $x \in E \setminus E_0$. 
				
				For the case $x \in E_0$ we observe that $Q^{t-k}(\mathds{1})(x) = Q^{t-k}(f)(x) =0$, which completes the proof. 
			\end{proof}
			
			\begin{lemma}\label{lem:stability_particle_approx_1}
				With the assumptions and notations as in Theorem~\ref{thm:stability_particle_approx}, there exists a constant $B\in (0, \infty)$ such that for any $t, k \in \N$ with $k \leq t$ and any bounded and measurable $f$ we have  
				\[
				\E\left[ \mathds{1}_{A_t} \left\vert \frac{\eta_{k}^N Q^{t-k} (f_t)}{\eta_{k}^N Q^{t-k}(\mathds{1})} - \frac{\eta_{k}^N Q^{t-k} (f_t)}{\Phi(\eta_{k-1}^N) Q^{t-k}(\mathds{1})} \right\vert^p \right]^{1/p} \leq \frac{B \tau^{t-k} \norm{f}}{\sqrt{N}}.
				\]
			\end{lemma}
			\begin{proof}
				For the sake of brevity write $a = \eta_{k}^N Q^{t-k}(\mathds{1})$ and $b= \Phi(\eta_{k-1}^N) Q^{t-k}(\mathds{1})$. 
				Using Lemma \ref{lem:bound_semigroup} we deduce that on the event $\{\tau(N)> k \}$ we have
				\begin{align*}
					\vert \eta_{k}^N Q^{t-k} (f_t) \vert \left\vert \frac{1}{a} - \frac{1}{b} \right\vert  &\leq 
					C \tau^{t-k} \norm{f} \eta_{k}^N Q^{t-k}(\mathds{1})  \left\vert \frac{1}{a} - \frac{1}{b} \right\vert \\
					&= 
					C \tau^{t-k} \norm{f}\frac{1}{\keff^{-t+k}b} \, \left\vert \keff^{-t+k}a - \keff^{-t+k}b\right\vert.
				\end{align*}  
				Additionally, since $b$ and $\mathds{1}_{\{ \tau(N) > k \}}$ are $\mathcal{F}_{k-1}$-measurable, Assumption \ref{assumption_E_p} implies 
				\[
				\E\left[ \mathds{1}_{\{ \tau(N) > k \}}\left( \frac{1}{\keff^{-t+k}b} \, \left\vert \keff^{-t+k}a - \keff^{-t+k}b \right\vert\right)^p \bigg\vert \mathcal{F}_{k-1}\right]^{1/p} \leq \frac{\mathds{1}_{\{ \tau(N) > k \}}}{\keff^{-t+k}b} \frac{C_p \norm{\omega_{t-k}}}{\sqrt{N}},
				\]
				where $\omega_{t-k} = \keff^{-t+k} Q^{t-k}(\mathds{1})$.
				
				On the event ${\{ \tau(N) > k \}}$ we have 
				\[
				\frac{1}{\keff^{-t+k}b} = \frac{1}{\Phi(\eta_{k-1}^N)(\omega_{t-k})} = \frac{\sum_{j=1}^N G(\xi_j^t)}{\sum_{j=1}^N G(\xi_j^t)M\omega_{t-k}(\xi_j^t)}.
				\]
				Note that $\omega_{n} \to \omega$ uniformly on $E \setminus E_0$, cf. Assumption \ref{ass_perron_frobenius_assumption}, so that 
				\[
				M\omega_{t-k}(x) \geq \int_{E \setminus E_0} \omega_{t-k}(x') M(x, \d x') \geq c \inf_{y \in E \setminus E_0}\omega(y) =: \widetilde{c} >0,
				\]
				for some constant $c >0$ independent of $t, k, N$.
				Hence, 
				\[
				\frac{\mathds{1}_{\{ \tau(N) > k \}}}{\keff^{-t+k}b} \leq \frac{1}{\widetilde{c}}.
				\]
				By the uniform convergence of $\omega_{t-k}$, it follows that $\norm{\omega_{t-k}}$ is uniformly bounded (in $t,k,N$), and the claimed bound is a consequence of the estimates above. 
			\end{proof}
			
			\begin{lemma}\label{lem:stability_particle_approx_2}
				With the assumptions and notations as in Theorem~\ref{thm:stability_particle_approx}, there exists a constant $B\in (0, \infty)$ such that for any $t, k \in \N$ with $k \leq t$ and any bounded and measurable $f$ we have 
				\[
				\E\left[ \mathds{1}_{A_t} \left\vert \frac{\eta_{k-1}^N Q^{t-k+1} (f_t)}{\eta_{k-1}^N Q^{t-k+1}(\mathds{1})} - \frac{\eta_{k}^N Q^{t-k} (f_t)}{\Phi(\eta_{k-1}^N) Q^{t-k}(\mathds{1})} \right\vert^p \right]^{1/p} \leq \frac{B \tau^{t-k} \norm{f}}{\sqrt{N}}
				\]
				with some constant $B \in (0, \infty)$. 
			\end{lemma}
			\begin{proof}
				On the event $\{\tau(N) >k\}$, from the definition of $\Phi$, we have 
				\begin{align*}
					\left\vert \frac{\eta_{k-1}^N Q^{t-k+1} (f_t)}{\eta_{k-1}^N Q^{t-k+1}(\mathds{1})} - \frac{\eta_{k}^N Q^{t-k} (f_t)}{\Phi(\eta_{k-1}^N) Q^{t-k}(\mathds{1})} \right\vert
					=
					\left\vert \frac{\Phi(\eta_{k-1}^N) Q^{t-k} (f_t)}{\Phi(\eta_{k-1}^N) Q^{t-k}(\mathds{1})} - \frac{\eta_{k}^N Q^{t-k} (f_t)}{\Phi(\eta_{k-1}^N) Q^{t-k}(\mathds{1})}  \right\vert \\
					=
					\frac{1} {\Phi(\eta_{k-1}^N) Q^{t-k}(\mathds{1})} \left\vert {\Phi(\eta_{k-1}^N) Q^{t-k} (f_t)} - {\eta_{k}^N Q^{t-k} (f_t)}\right\vert .
				\end{align*}
				Using $\mathcal{F}_{k-1}$-measurability of $\mathds{1}_{\{\tau(N) > k\}}$ and $\Phi(\eta_{k-1}^N) Q^{n-k} (f_t)$, as well as Assumption~\ref{assumption_E_p} and Lemma \ref{lem:bound_semigroup} we deduce that 
				\begin{align*}
					&\E \left[ \mathds{1}_{\{\tau(N) > k\}} \left\vert \frac{\eta_{k-1}^N Q^{t-k+1} (f_t)}{\eta_{k-1}^N Q^{t-k+1}(\mathds{1})} - \frac{\eta_{k}^N Q^{t-k} (f_t)}{\Phi(\eta_{k-1}^N) Q^{t-k}(\mathds{1})} \right\vert^p \right]^{1/p} \\
					&\qquad\leq
					\frac{C_p \norm{Q^{t-k}(f_t)}}{\sqrt{N}} \; \E \left[ \mathds{1}_{\{\tau(N) > k\}}\left\vert \frac{1} {\Phi(\eta_{k-1}^N) Q^{t-k}(\mathds{1})} \right\vert^p  \right]^{1/p} \\
					&\qquad\leq
					\frac{C_p C \tau^{t-k} \norm{Q^{t-k}(\mathds{1})}}{\sqrt{N}} \; \E \left[ \mathds{1}_{\{\tau(N) > k\}}\left\vert \frac{1} {\Phi(\eta_{k-1}^N) Q^{t-k}(\mathds{1})} \right\vert^p  \right]^{1/p} \\
					&\qquad=
					\frac{C_p C \tau^{t-k} \norm{\keff^{-t+k}Q^{t-k}(\mathds{1})}}{\sqrt{N}} \; \E \left[ \mathds{1}_{\{\tau(N) > k\}}\left\vert \frac{1} {\Phi(\eta_{k-1}^N) (\keff^{-t+k}Q^{t-k}(\mathds{1}))} \right\vert^p  \right]^{1/p}.
				\end{align*}
				Again, we write $\omega_{t-k} = \keff^{-t+k}Q^{t-k}(\mathds{1})$.
				Using the uniform convergence of $\omega_{t-k}$, see Assumption \ref{ass_perron_frobenius_assumption}, it follows that $\norm{\omega_{t-k}}$ is bounded uniformly in $t,k,N$. 
				
				Using the same steps as in the proof of Lemma \ref{lem:stability_particle_approx_1} (see the bound for $1/(\keff^{t-k}b)$ there), it follows that 
				\[
				\E \left[ \mathds{1}_{\{\tau(N) > k\}}\left\vert \frac{1} {\Phi(\eta_{k-1}^N) (\keff^{-t+k}Q^{t-k}(\mathds{1}))} \right\vert^p  \right]^{1/p} \leq \frac{1}{\widetilde{c}} 
				\]
				for a constant $\widetilde{c}> 0$ independent of $t, k , N$. This implies the desired result. 
			\end{proof}
			
			\begin{proof}[Proof of Theorem \ref{thm:stability_particle_approx}]
				We estimate all terms appearing in the error decomposition \eqref{equ:error_decomp}.
				By Assumption~\ref{assumption_E_p} we obtain 
				\[
				\E\left[\mathds{1}_{A_t} \left\vert \eta_{t}^N(f_t) - \Phi(\eta_{t-1}^N)(f_t) \right\vert^p  \right]^{1/p} 
				\leq 
				\frac{C_p\norm{f_t}}{\sqrt{N}} \leq \frac{2C_p\norm{f}}{\sqrt{N}}.
				\]
				Next we bound $\sum_{k=1}^{t-1} \E\left[ \mathds{1}_{A_t} \left\vert \Phi^{t-k} (\eta_k^N) (f_t) - \Phi^{t-k+1} (\eta_{k-1}^N) (f_t)  \right\vert^p \right]^{1/p}$. 
				By Lemma \ref{lem:stability_particle_approx_1} and Lemma \ref{lem:stability_particle_approx_2} it follows 
				\begin{align*}
					&\E\left[ \mathds{1}_{A_t} \left\vert \Phi^{t-k} (\eta_k^N) (f_t) - \Phi^{t-k+1} (\eta_{k-1}^N) (f_t)  \right\vert^p \right]^{1/p} \\
					& \qquad \leq \E\left[ \mathds{1}_{A_t} \left\vert \frac{\eta_{k}^N Q^{t-k} (f_t)}{\eta_{k}^N Q^{t-k}(\mathds{1})} - \frac{\eta_{k}^N Q^{t-k} (f_t)}{\Phi(\eta_{k-1}^N) Q^{n-k}(\mathds{1})} \right\vert^p \right]^{1/p}\\
					&\qquad\qquad+	
					\E\left[ \mathds{1}_{A_t} \left\vert \frac{\eta_{k-1}^N Q^{t-k+1} (f_t)}{\eta_{k-1}^N Q^{t-k+1}(\mathds{1})} - \frac{\eta_{k}^N Q^{t-k} (f_t)}{\Phi(\eta_{k-1}^N) Q^{t-k}(\mathds{1})} \right\vert^p \right]^{1/p}\\
					&\qquad\leq
					\frac{\widetilde{B} \norm{f} \tau^{t-k}}{\sqrt{N}},
				\end{align*}
				with $\widetilde{B} \in (0, \infty)$ being some constant independent of $t, k ,N$.
				Hence 
				\[
				\sum_{k=1}^{t-1} \E\left[ \mathds{1}_{A_t} \left\vert \Phi^{t-k} (\eta_k^N) (f_t) - \Phi^{t-k+1} (\eta_{k-1}^N) (f_t)  \right\vert^p \right]^{1/p}
				\leq \frac{\widetilde{B}}{1- \tau} \frac{\norm{f}}{\sqrt{N}}.
				\]
				
				Finally, we turn to the third term in \eqref{equ:error_decomp}. Note that 
				\[
				\E\left[ \mathds{1}_{A_t} \left\vert \Phi^t(\eta_0^N)(f_t) \right\vert^p  \right]^{1/p} 
				\leq 
				\E\left[ \left\vert \frac{\eta_0^N Q^t(f_t)}{\eta_0^N Q^t(\mathds{1})} - \frac{\eta_0^N Q^t(f_t)}{\mu_0 Q^t(\mathds{1})} \right\vert^p  \right]^{1/p} + \E\left[  \left \vert \frac{\eta_0^N Q^t(f_t)}{\mu_0 Q^t(\mathds{1})} \right\vert^p \right]^{1/p}.
				\]
				By Lemma \ref{lem:bound_semigroup}, and since $\eta_0^N$ is an iid approximation of $\mu_0$, we have
				\begin{align*}
					\E\left[ \left\vert \frac{\eta_0^N Q^t(f_t)}{\eta_0^N Q^t(\mathds{1})} - \frac{\eta_0^N Q^t(f_t)}{\mu_0 Q^t(\mathds{1})} \right\vert^p  \right]^{1/p} 
					\leq 
					C \tau^t \norm{f} \E\left[ \eta_0^NQ^t(\mathds{1}) \left\vert \frac{1}{\eta_0^N Q^t(\mathds{1})} - \frac{1}{\mu_0 Q^t(\mathds{1})} \right\vert^p  \right]^{1/p} \\
					\leq 
					C \tau^t \norm{f} \E\left[ \left\vert 1 - \frac{\eta_0^N Q^t(\mathds{1})}{\mu_0 Q^t(\mathds{1})} \right\vert^p  \right]^{1/p} 
					=
					\frac{C \tau^t \norm{f} }{\mu_0 Q^t(\mathds{1})} \E\left[ \left\vert {\eta_0^N Q^t(\mathds{1})} - {\mu_0 Q^t(\mathds{1})} \right\vert^p  \right]^{1/p}\\
					\leq 
					\frac{C' \tau^t \norm{f} }{\sqrt{N}} \frac{\norm{Q^t(\mathds{1})}}{{\mu_0 Q^t(\mathds{1})}} 
					= 
					\frac{C' \tau^t \norm{f} }{\sqrt{N}} \frac{\keff^{-t}\norm{Q^t(\mathds{1})}}{{\keff^{-t}\mu_0 Q^t(\mathds{1})}}.
				\end{align*}
				As in the proof of Lemma \ref{lem:stability_particle_approx_1} we can use the uniform convergence of $\omega_{t}$ to deduce that $\frac{\keff^{-t}\norm{Q^t(\mathds{1})}}{{\keff^{-t}\mu_0 Q^t(\mathds{1})}}$ is uniformly bounded in $t, N$.

				Finally, since $\mu_0 Q^t (f_t) =\Phi^t (\mu_0)(f) - \eta_t(f)  =0$ we have by Lemma \ref{lem:bound_semigroup} and similar arguments as above,  
				\begin{align*}
					\E\left[  \left \vert \frac{\eta_0^N Q^t(f_t)}{\mu_0 Q^t(\mathds{1})} \right\vert^p \right]^{\frac{1}{p}} 
					= 
					\frac{1}{\mu_0 Q^t(\mathds{1})}\E\left[  \left \vert {\eta_0^N Q^t(f_t)} - \mu_0 Q^t (f_t) \right\vert^p \right]^{\frac{1}{p}}  
					\leq \frac{C' \tau^t}{\sqrt{N}} \frac{\norm{Q^t(\mathds{1})}}{{\mu_0 Q^t(\mathds{1})}} \leq \frac{\widehat{C} \tau^t}{\sqrt{N}}.
				\end{align*} 
				Hence, the right hand side in \eqref{equ:error_decomp} is bounded by $\frac{\norm{f} B}{\sqrt{N}}$, for some constant $B< \infty$ independent of $t,N$ and $f$, which shows the claimed result. 
			\end{proof}

				\begin{remark}
					Theorem~\ref{thm:stability_particle_approx} can also be used to derive
					robust estimates for the fluence $\varphi$. The fluence is the positive eigenmeasure of the $\keff$ equation \eqref{equ:keff_problem}. That is, 
					\[
					\varphi \left(\left(  {\mathcal T} + \mathcal{S} + \frac{1}{\keff}{\mathcal F}\right) g\right) =0
					\]
					for any bounded and measurable $g \colon E \to \R$. The physical interpretation of $\varphi$ is that it quantifies the (relative) expected number of particles moving through different regions between fission events. The closely related measure $\eta(\d x) = \frac{1}{\keff} \varphi \mathcal{F}  (\d x)$ then describes where particles are born, and is exactly the $\eta$ given above. Furthermore, if for an arbitrary bounded and measurable $g \colon E \to \R$ we set $\tilde{g} = -(\mathcal{T} + \mathcal{S})g$, then 
					\[
					\eta( -(\mathcal{T} + \mathcal{S})^{-1}\tilde{g}) = \varphi (\tilde{g}).
					\]
					By standard results (see e.g.~the discussion in Section~2.4 of \cite{SNTE-III}) we have 
					\[
					-(\mathcal{T} + \mathcal{S})^{-1}\tilde{g} (x) = \E_{x} \left[ \int_0^{\tau_F} \tilde{g}((r(s),\upsilon(s))) \d s \right],
					\]
					where $(r(s),\upsilon(s))$ is the stochastic representation of the underlying particles with $(r(0),\upsilon(0)) = x$, and $\tau_F$ is the first time a particle either undergoes fission, or leaves the system. It is therefore possible to use the particle measures $\eta_{k}^N$ to compute estimates for the fluence $\varphi(g)$ and therefore $\varphi$ by incorporating $\tilde{g}$ as above. In particular, this is the process commonly referred to as \emph{tallying} in the nuclear engineering literature.
				\end{remark}

				\subsubsection{Arithmetic mean estimators}
				
				In this section we study $\kfrak_{T, k}^{ari}$.

				\begin{theorem}\label{thm:arithmetic_mean_error_bound}
					Let Assumptions \ref{ass_perron_frobenius_assumption}, \ref{ass:minimal_survival_probability} and \ref{assumption_E_p} (for some $p \in [1, \infty)$) be true.
					Then, there exists a constant $K_p \in (0, \infty)$ such that for any $T,k,N \in \N$ we have 
					\[
					\E\left[ \left\vert\kfrak_{T, k}^{ari} - \keff  \right\vert^p  \right]^{1/p} \leq K_p \left( \frac{\norm{G}}{\sqrt{N}} + \frac{\tau^{T}}{k(1-\tau)} + (T+k+1)^{1/p}e^{-(a N)/p}\right) , 
					\] 
					with $a = - \log(1-\kappa) \in (0, \infty)$, where $\kappa$ is as in Assumption \ref{ass:minimal_survival_probability}.
				\end{theorem}
				
				\begin{proof}
					Using Proposition \ref{prop:exponential_congergence_in_total_variation}, and integrating over $E\setminus E_0$, we obtain that 
					\[
					\left\vert \eta_t(G) - \keff  \right\vert = \left\vert \Phi^t(\eta_0)(G) - \eta(G)  \right\vert \leq \widetilde{C}\tau^t,
					\]
					for any $t \in \N$. Hence, by Theorem \ref{thm:stability_particle_approx} it follows that for any $t \in \N$ we have 
					\[
					\E\left[\mathds{1}_{\{\tau(N)> t\}} \left\vert \eta_{t}^N(G) - \keff  \right\vert^p  \right]^{1/p} \leq \frac{K_p' \norm{G}}{\sqrt{N}} + \widetilde{C}\tau^t.
					\]
					Recall, that if the system has died out, we estimate $\keff$ with $0$. Thus, by the above estimates and Proposition \ref{prop:survial_probability_estimate} it follows
					\begin{align*}
						\E\left[\left\vert \eta_t^N(G) - \keff  \right\vert^p\right]^{1/p} 
						&\leq 
						\E\left[\mathds{1}_{\{\tau(N)> t\}} \left\vert \eta_{t}^N(G) - \keff  \right\vert^p  \right]^{1/p} + \keff \mathbb{P}[\tau(N) \leq t]^{1/p} \\
						& \leq
						\frac{K_p' \norm{G}}{\sqrt{N}} + \widetilde{C}\tau^t + \keff (t e^{-a N})^{1/p}.
					\end{align*}
					Using the triangle inequality we obtain that 
					\begin{align*}
						\E [\vert \kfrak_{T, k}^{ari} - \keff \vert^p] 
						&\leq 
						\frac{K_p' \norm{f}}{\sqrt{N}} + \frac{1}{k}\sum_{t=T+1}^{T+k}\widetilde{C}\tau^t + \frac{1}{k}\sum_{t=T+1}^{T+k}\keff (t e^{-a N})^{1/p} \\
						&\leq 
						\frac{K_p' \norm{f}}{\sqrt{N}} + \frac{\tau^T \widetilde{C}\tau^t}{k (1-\tau)} + \frac{\keff e^{-aN/p}}{k}\sum_{t=T+1}^{T+k} t^{1/p}.
					\end{align*}
					By the mean value theorem (in integral form) there exists some $x_0 \in [T,T+k+1]$ such that 
					\begin{align*}
						\frac{1}{k}\sum_{t=T+1}^{T+k} t^{1/p} 
						\leq 
						\frac{1}{k}\int_{T}^{T+k+1} x^{1/p} \d x = \frac{k+1}{k} x_0^{1/p} \leq 2 (T+k+1)^{1/p}.
					\end{align*}
					Combining the above inequalities shows the claimed result. 
				\end{proof}
				
				From the proof of Theorem \ref{thm:arithmetic_mean_error_bound} we immediately obtain the following. 
				\begin{corollary}
					Let Assumptions \ref{ass_perron_frobenius_assumption}, \ref{ass:minimal_survival_probability} and \ref{assumption_E_p} (for some $p \in [1, \infty)$) be true.
					Then, there exists a constant $K_p \in (0, \infty)$ such that for any $T,k,N \in \N$ we have 
					\[
					\E\left[\mathds{1}_{\{\tau(N)>t\}} \left\vert\kfrak_{T, k}^{ari} - \keff  \right\vert^p  \right]^{1/p} \leq K_p \left( \frac{\norm{G}}{\sqrt{N}} + \frac{\tau^{T}}{k(1-\tau)}\right) .
					\] 
				\end{corollary}
				
				Finally, we observe that in practice, a typical use case would be to
				run a Monte Carlo simulation to estimate $\keff$, which chooses the
				parameters $T, k, N$, which aims to trade off two sources of loss:
				first, we want to ensure that with high probability our simulation
				does not go extinct by the time $T+k$. Second, that our computation
				returns an estimate $\kfrak_{T,k}^{ari}$ which has small enough
				tolerance, conditional on the calculation terminating. We formulate
				this in the following result.
				
				\begin{theorem} \label{thm_conditional_convergence}
					Suppose Assumptions \ref{ass_perron_frobenius_assumption},
					\ref{ass:minimal_survival_probability} and \ref{assumption_E_p} (for
					some $p \in [1, \infty)$) are true. Then, there exists some constant $\tilde{K}_p \in (0, \infty)$ such
					that for any $T, k \in \N$ there exists $N_0 \in \N$ such that for all
					$N \ge N_0$ we have
					\[
					\E\left[\left\vert\kfrak_{T, k}^{ari} - \keff  \right\vert^p \big|\tau(N)>T+k  \right]^{1/p} \leq \tilde{K}_p \left( \frac{\norm{G}}{\sqrt{N}} + \frac{\tau^{T}}{k(1-\tau)}\right),
					\]
					and
					\[
					\P\left(\tau(N)<T+k  \right) \le (T+k) \me^{-aN},
					\]
					with $a = - \log(1-\kappa) \in (0, \infty)$, where $\kappa$ is as in Assumption \ref{ass:minimal_survival_probability}.
					
					Moreover, given $\eps, \delta > 0$, we can find $T,k,N$ sufficiently
					large that
					\[
					\E\left[\left\vert\kfrak_{T, k}^{ari} - \keff  \right\vert^p
					\big|\tau(N)>T+k  \right]^{1/p} < \eps
					\text{ and }
					\P\left(\tau(N)<T+k  \right) \le \delta.
					\]
				\end{theorem}
				
				\begin{proof}
					By Theorem~\ref{thm:arithmetic_mean_error_bound}, for fixed $T, k$,
					it is clear that we can choose $N_0$ sufficiently large that
					$\P\left(\tau(N)<T+k \right) \le \min\{\delta,\frac{1}{2}\}$. Taking
					$\tilde{K}_p = K_p 2^{1/p}$, we have the first claim. Moreover,
					given $\eps > 0$, it is now immediate that by fixing $T$ and taking
					$k$ and $N$ sufficiently large, the rest of the statement follows.
				\end{proof}
				
				\begin{remark}
					We expect that under suitable mixing assumptions on the particles $(\xi_1^t, \dots , \xi_N^t)_{ t \in \N_0}$ it is possible to obtain $(kN)^{-1/2}$, rather than $N^{-1/2}$ in the results above. This improved bound would also take into account that we use a arithmetic mean, however, we leave a further investigation open for future research.  
				\end{remark}
				
				\begin{remark} \label{rk-cost-calculations}
					It is natural to ask: for a given error tolerance $\eps >0$ and
					failure probability $\delta$, what is the minimal computational
					requirement?
					
					We suppose that the cost of the calculation scales with the number
					of particles per generation, and the number of generations
					simulated, so that the total numerical cost is proportional to
					$N(T+k)$. We can formulate this as an optimisation problem::
					\begin{align*}
						\text{minimise } N(T+k) \qquad \text{ subject to: } &  \left(
						\frac{\norm{G}}{\sqrt{N}} + \frac{\tau^{T}}{k(1-\tau)}\right)
						\le \eps\\
						& (T+k) \me^{-aN} \le \delta,\\
						& T, k, N \ge 0.
					\end{align*}
					Here we relax the practical assumption that $T,k,N$ are integers for
					formal analysis, and note that in addition we have some constants,
					independent of $N, T, k, \norm{G}, \tau, a$ which we absorb into
					$\eps, \delta$.  To simplify notation, we write
					$\tau = \me^{-\beta}$, where $\beta > 0$ and $\gamma:= \norm{G}^2$.
					
					In this regime, we normally expect the pertinent constraint to come
					from the first term, and in this case, ignoring the second
					constraint, we compute that the optimum occurs when:
					\begin{align*}
						k^* & = \beta^{-1}; \qquad 		&T^*  = -\frac{1}{\beta} \left( W_{-1}\left( -\frac{\me \gamma
							(1-\me^{-\beta})}{2 \beta \sqrt{N^*}}\right) \right);\\
						N^* &  = \eps^{-2} \gamma^2 (1-w^*)^{-2};
						&w^*  = (3 + 2\beta T^*)^{-1} \in \left(0,\frac{1}{3}\right];
					\end{align*}
					where $W_{-1}(x)$ for $x \in [-\me^{-1},0)$ is the negative lower
					branch of the Lambert-W function\footnote{Specifically the unique solution
						$y$ to the equation $y \me^y = x$ with $y \le -1$.}. We leave the
					details of the calculation to the interested reader. 
					
					We observe that, up to a constant, (ignoring the
					multiplicative constants independent of $\gamma, \beta$ which we have
					suppressed in writing the above inequalities), we have:
					\begin{align*}
						\gamma^2 \eps^{-2} & \le N^* \le \frac{9}{4} \gamma^2 \eps^{-2}
					\end{align*}
					and we have total cost:
					\begin{align*}
						(T^* + k^*)N^*&  = \gamma^2 \eps^{-2} (1-w^*)^{-2} \beta^{-1} \left(
						\ln(\eps^{-1}) + o(\ln(\eps^{-1}))\right).
					\end{align*}
					In particular, for a given problem, the error behaves like
					$\eps^{-2} \ln (\eps^{-1})$ for $\eps$ small. In addition, for
					problems where the spectral gap is small, so $\beta>0$ is small, then
					the cost grows like $\beta^{-1}$. Moreover we note that the optimal number of steps to average over
					($k^*$) is determined by the spectral gap: when $\beta$ is smaller,
					the optimal number of generations to average over increases.
				\end{remark}

				\subsubsection{Geometric mean estimators}
				In this section we study $\kfrak_{T,k}^{geo}$.
				
				\begin{lemma}\label{lem:Assumption_B_q_true}
					Suppose Assumptions \ref{ass_perron_frobenius_assumption}, \ref{ass:minimal_survival_probability} be true. 
					Then, for any $q \in [1, \infty)$ there exists a constant $C_q$ such that for any $t, N \in \N$ we have
					\[
					\E\left[ \mathds{1}_{\{\tau(N) >t\}} \left( \frac{1}{\eta_{t}^N(G)} \right)^q \right] \leq C_q.
					\]
				\end{lemma}
				\begin{proof}
					Fix $t \in \N$ and $N \in \N$. 
					Define the sets $B_{\ell} = B_{\ell, t, N}$ as the events that exactly $\ell \in \{1, \dots, N\}$ particles survive from iteration $t$ to $t+1$ in the particle algorithm. 
					Then we can rewrite $\{\tau(N)>t\}$ as the disjoint union
					\[
					\{\tau(n)> t\} = \bigcup_{\ell=1}^N B_{\ell}.
					\]
					Since $G $ is uniformly lower bounded by $\kappa$ on $E\setminus E_0$ and `survivors' have to be in this region while particles that die are in $E_0$, it follows $\eta_{t}^{N}(G) \geq \frac{\ell}{N}$ on the event $B_\ell$. Thus
					\begin{align*}
						\E\left[ \mathds{1}_{\{\tau(N) >t\}} \left( \frac{1}{\eta_{t}^N(G)} \right)^q \right] \\
						\leq 
						\sum_{\ell=1}^{N}\E\left[ \mathds{1}_{B_\ell} \left( \frac{N}{\kappa \ell} \right)^q \right] 
						= 
						\frac{N^q}{\kappa^q }\sum_{\ell=1}^{N} \frac{1}{\ell^q}\P \left[ {B_\ell}  \right]  .
					\end{align*}
					Note that particles survive with a minimal probability of $\kappa>0$, and each particle does so independent of all other particles. 
					Hence, 
					\[
					\mathbb{P}[B_\ell] \leq \mathbb{P}[Y = \ell],
					\]
					where $Y \sim \text{Binomial}(N, \kappa)$. 
					Define the function $h_q \colon \{0, \dots, N\} \to [0, \infty]$ via $h_q(n) = 	0$ if $n=0$ and $
					h(n)=	\frac{1}{n^q}$ else.
					Then we have 
					\[
					\sum_{\ell=1}^{N} \frac{1}{\ell^q}\P \left[ {B_\ell}  \right] 
					\leq 
					\E[h_q(Y)] \leq \frac{C_q}{N^q},
					\]
					by using standard estimates for binomially distributed random variables. 
				\end{proof}
				
				\begin{lemma}\label{lem:fk-normalised-lower-bound}
					Under Assumption \ref{ass_perron_frobenius_assumption} and \ref{ass:minimal_survival_probability} we have $L := \inf_{n \in \N} \normalisedIntegral{n}{G} > 0$.
				\end{lemma}
				\begin{proof}
					By Assumption \ref{ass_perron_frobenius_assumption}, $\lim_{n \rightarrow \infty} \normalisedIntegral{n}{G} = \lim_{n \rightarrow \infty} \Phi^n \left( \eta_n \right) (G) = \keff >0$. Hence, the result follows by well known properties of converging sequences of positive real numbers. 
				\end{proof}
				
				\begin{theorem}\label{thm:geometric_mean_error_bound}
					Let Assumptions \ref{ass_perron_frobenius_assumption}, \ref{ass:minimal_survival_probability} and \ref{assumption_E_p} (for some $p' \in (1, \infty)$) be true.
					Then, for any $p \in [1, p')$, there exists some constant $K_{p, p'} \in (0, \infty)$ such that for any $T,k,N \in \N$ we have 
					\[
					\E\left[ \left\vert\kfrak_{T, k}^{geo} - \keff  \right\vert^p  \right]^{1/p} \leq K_{p,p'} \left( \frac{\norm{G}^2}{\sqrt{N}} + \frac{\tau^T }{k} +  (T+k+1)^{1/p}e^{-(a N)/p}\right) , 
					\] 
					with $a = - \log(1-\kappa) \in (0, \infty)$, where $\kappa$ is as in Assumption \ref{ass:minimal_survival_probability}.
				\end{theorem}
				
				\begin{proof}
					Define $\gamma_{T,k} = \left( \prod_{\ell=T+1}^{T+k} \eta_\ell(G)\right)^{1/k} $ and $A = \left\{ \kfrak_{T,k}^{geo} \geq \gamma_{T,k} \right\} \in \mathcal{F}$. On the event $\left\{ {\tau}({N}) > T+k \right\} \cap A$, we have
					\begin{align*}
						\left| {\kfrak_{T,k}^{geo}} - \gamma_{T,k} \right| &= \gamma_{T,k} \left( \frac{ {\kfrak_{T,k}^{geo}} }{ \gamma_{T,k} } - 1 \right)
						\leq 
						\norm{G} \left[ \left( \prod_{\ell=T+1}^{T+k} \frac{ \eta^{N}_{\ell}({G}) }{ \normalisedIntegral{\ell}{G} }\right)^\frac1{k} - 1 \right] \\[12pt]
						&\leq 
						\frac{ \norm{G} }{k} \sum_{\ell=T+1}^{T+k} \left( \frac{ \eta^{N}_{\ell}({G}) }{ \normalisedIntegral{\ell}{G} } - 1 \right) 
						= 
						\frac{ \norm{G} }{k} \sum_{\ell=T+1}^{T+k} \frac1{\normalisedIntegral{\ell}{G}} \left( \eta^{N}_{\ell}({G}) - \normalisedIntegral{\ell}{G} \right) \\[12pt]
						&\leq 
						\frac{ \norm{G} }{kL} \sum_{\ell=T+1}^{T+k} \left| \eta^{N}_{\ell}({G}) - \normalisedIntegral{\ell}{G} \right|,
					\end{align*}
					where third line follows from Lemma \ref{lem:fk-normalised-lower-bound}. 
					
					Similarly, since $\left\{ {\tau}({N}) > T+k \right\} \subseteq \left\{ \particleintegral{N}{\ell}{G} > 0 \right\}$ for $\ell \in \left\{ 0, \ldots, T+k \right\}$, the following holds on the event $\left\{ {\tau}({N}) \geq T+k \right\} \cap A^c$:
					\begin{align*}
						\left| \kfrak_{T,k}^{geo} - \gamma_{T,k} \right| 
						&= 
						\kfrak_{T,k}^{geo} \left( \frac{ \gamma_{T,k} }{ \kfrak_{T,k}^{geo}} -1 \right) 
						= 
						\kfrak_{T,k}^{geo} \left[ \left( \prod_{\ell=T+1}^{T+k} \frac{ \normalisedIntegral{\ell}{G} }{ \eta^{N}_{\ell}({G}) } \right)^\frac1{k} - 1 \right] \\[12pt]
						&\leq 
						\frac{ \norm{G} }{k} \sum_{\ell=T+1}^{T+k} \left( \frac{ \normalisedIntegral{\ell}{G} }{ \eta^{N}_{\ell}({G}) } - 1 \right) 
						\leq 
						\frac{ \norm{G} }{k} \sum_{\ell=T+1}^{T+k} \frac1{ \particleintegral{N}{\ell}{G}} \left| \eta^{N}_{\ell}({G}) - \normalisedIntegral{\ell}{G} \right|.
					\end{align*}
					Using these estimates, and abbreviating $\mathcal{C}= \{\tau({N}) > T+k\}$, we obtain  
					\begin{align*} 
						\expectationOperator{ \mathds{1}_{\mathcal{C}} \left| {\kfrak_{T,k}^{geo}} - \gamma_{T,k} \right|^p }^\frac1{p}  
						\leq \expectationOperator{ \mathds{1}_{\mathcal{C}} \left| \frac{ \norm{G} }{kL} \sum_{\ell=T+1}^{T+k} \left| \particleintegral{N}{\ell}{G} - \normalisedIntegral{\ell}{G} \right| \right|^p }^\frac1{p} \\
						\qquad+ \expectationOperator{ \mathds{1}_{\mathcal{C}} \left| \frac{ \norm{G} }{ k } \sum_{\ell=T+1}^{T+k} \frac1{ \particleintegral{N}{\ell}{G} } \left| \particleintegral{N}{\ell}{G} - \normalisedIntegral{\ell}{G} \right| \right|^p }^\frac1{p}.
					\end{align*}
					As a consequence of Jensen's inequality (for sums), Minkowski's inequality and Theorem \ref{thm:stability_particle_approx} (note that if Assumption \ref{assumption_E_p} is true for $p'$, then also for $p< p'$) it follows 
					\[
					\expectationOperator{ \mathds{1}_{\{\tau({N}) > T+k\}} \left| \frac{ \norm{G} }{kL} \sum_{\ell=T+1}^{T+k} \left| \particleintegral{N}{\ell}{G} - \normalisedIntegral{\ell}{G} \right| \right|^p }^\frac1{p}
					\leq 
					\frac{K_p\norm{G}^2}{kL \sqrt{N}}.
					\]
					Next, let $q =p'/p >1$ and denote by $q'$ the H\"{o}lder conjugate of $q$. Again using Jensen and Minkowski inequalities, followed by H\"{o}lder's inequality and Theorem \ref{thm:stability_particle_approx} together with Lemma \ref{lem:Assumption_B_q_true} and $\mathcal{C} = \{\tau({N}) > T+k\}$, it follows that 
					\begin{align*}
						\expectationOperator{ \mathds{1}_{\mathcal{C}} \left| \frac{ \norm{G} }{ k } \sum_{\ell=T+1}^{T+k} \frac1{ \particleintegral{N}{\ell}{G} } \left| \particleintegral{N}{\ell}{G} - \normalisedIntegral{\ell}{G} \right| \right|^p }^\frac1{p} \\
						\leq 
						\frac{ \norm{G} }{ k } \sum_{\ell=T+1}^{T+k}\expectationOperator{ \mathds{1}_{\mathcal{C}} \left( \frac1{ \particleintegral{N}{\ell}{G} } \left| \particleintegral{N}{\ell}{G} - \normalisedIntegral{\ell}{G} \right| \right)^p }^\frac1{p} 
						\leq 
						\frac{\norm{G}^2C_{p, p'}}{\sqrt{N}},
					\end{align*}
					with a constant $C_{p, p'}$ only depending on $p, p'$. 
					Hence, 
					\begin{equation}\label{equ:geometric_estimator_bound_survival}
						\expectationOperator{ \mathds{1}_{\tau({N}) > T+k} \left| {\kfrak_{T,k}^{geo}} - \gamma_{T,k} \right|^p }^\frac1{p}
						\leq 
						\frac{\norm{G}^2 \widetilde{C}_{p,p'}}{L\sqrt{N}},
					\end{equation}
					where $\widetilde{C}_{p, p'}$ only depends on $p, p'$. 
					
					Next we estimate $\vert \gamma_{T,k} - \keff\vert$. If $\gamma_{T,k} \geq \keff$ then 
					\begin{align*}
						\vert \gamma_{T,k} - \keff\vert = \keff \left( \prod_{\ell=T+1}^{T+k}\left( \frac{\eta_\ell(G)}{\keff}\right)^{1/k}  - 1\right)
						\leq 
						\frac{\keff}{k} \left( \sum_{\ell=T+1}^{T+k} \frac{\eta_\ell(G)}{\keff}  - 1\right)\\
						\leq \frac{1}{k}\sum_{\ell=T+1}^{T+k}\vert \eta_\ell(G) - \keff \vert.
					\end{align*}
					Analogously, if $\gamma_{T,k} < \keff$, then 
					\[
					\vert \gamma_{T,k} - \keff\vert= \gamma_{T,k}\left( \left( \prod_{\ell=T+1}^{T+k}\frac{\keff}{\eta_\ell(G)} \right)^{1/k} -1\right) 
					\leq 
					\frac{1}{k}\sum_{\ell=T+1}^{T+k} \vert \eta_\ell(G) - \keff \vert
					\]
					Hence, by combining these estimates with Assumption \ref{ass_perron_frobenius_assumption} (compare to the proof of Theorem \ref{thm:arithmetic_mean_error_bound}) we obtain 
					\[
					\vert \gamma_{T,k} - \keff\vert \leq \frac{1}{k}\sum_{\ell=T+1}^{T+k} \vert \eta_\ell(G) - \keff \vert \leq \frac{C \tau^T }{k},
					\]
					for a constant that is independent of $T,k,N$.
					Plugging into \eqref{equ:geometric_estimator_bound_survival} and using Minkowski's inequality, we deduce
					\begin{equation}\label{equ:geometric_estimator_bound_survival_2}
						\expectationOperator{ \mathds{1}_{\{\tau({N}) > T+k\}} \left| {\kfrak_{T,k}^{geo}} - \keff \right|^p }^\frac1{p}
						\leq 
						\frac{\norm{G}^2 C_{p,p'}'}{\sqrt{N}} + \frac{C \tau^T}{k(1-\tau)}.
					\end{equation}
					Finally, we recall that for a system that has died out, we estimate $\keff$ with $0$. 
					Thus, by \eqref{equ:geometric_estimator_bound_survival_2} and using Proposition \ref{prop:survial_probability_estimate} as in the proof of Theorem \ref{thm:arithmetic_mean_error_bound}, we arrive at 
					\begin{align*}
						\expectationOperator{  \left| {\kfrak_{T,k}^{geo}} - \keff \right|^p }^\frac1{p} 
						&= 
						\expectationOperator{ \mathds{1}_{\{\tau({N}) > T+k\}} \left| {\kfrak_{T,k}^{geo}} - \keff \right|^p }^\frac1{p} + \keff\mathbb{P}[\tau(N) \leq T+k]^{1/p} \\[10pt]
						&\leq
						C_{p,p'}'' \left(\frac{\norm{G}^2}{\sqrt{N}} + \frac{\tau^T }{k} +  (T+k+1)^{1/p}e^{-(a N)/p}\right),
					\end{align*}
					for a constant $C_{p,p'}''$ that only depends on $p,p'$.
				\end{proof}
				Note that \eqref{equ:geometric_estimator_bound_survival_2} in the
				above proof also gives the corollary below.
				
				\begin{corollary}
					L Assumptions \ref{ass_perron_frobenius_assumption}, \ref{ass:minimal_survival_probability} and \ref{assumption_E_p} (for some $p' \in (1, \infty)$) be true.
					Then, for any $p \in [1,p')$ there exists some constant $K_{p,p'} \in (0, \infty)$ such that for any $T,k,N \in \N$ we have 
					\[
					\E\left[\mathds{1}_{\{\tau(N)>T+k\}} \left\vert\kfrak_{T, k}^{geo} - \keff  \right\vert^p  \right]^{1/p} \leq K_{p,p'} \left( \frac{\norm{G}^2}{\sqrt{N}} + \frac{\tau^{T}}{k(1-\tau)}\right) .
					\] 
				\end{corollary}
				
				We also observe that, with essentially identical arguments and conclusions,
				the above result can be modified to give a version of
				Theorem~\ref{thm_conditional_convergence} and Remark~\ref{rk-cost-calculations} in the geometric estimator context.
				
				\section{Power Iteration Algorithms}\label{sec:PWR_Iteration}
				In this section we apply our results to Monte Carlo power iteration algorithms. First, we study the \textit{explicit fission} method in Section \ref{subsec:explicit_fission}. After that we study variations of the algorithm that are used in practice for variance reduction purposes. 
				
				\subsection{Explicit Fission}\label{subsec:explicit_fission}
				We remind ourselves about the roles of the potential function $G$ and the Markov kernel $M$. Namely, $G(x)$ is the number of expected offspring for a particle that is about to undergo fission in configuration $x$, while $M$ models the evolution between fission events. 
				
				Throughout the section, we consider physical systems which fit into the setting of Section \ref{subsec:Physical_assumptions}. Thus, Assumptions \ref{assumption:Q-perron-frobenius} and \ref{ass:minimal_survival_probability} are satisfied.

				The explicit fission algorithm, see Algorithm \ref{alg:genericPI}, is used to implement (a population controlled) system of $N\in \N$ neutrons. 
				
				\begin{algorithm}\label{alg:genericPI}
					\caption{`Explicit fission' power iteration algorithm}
					\begin{algorithmic}[1]
						\State \textbf{Input:}			
						\State $T$: number of generations for burn in
						\State $k$: number of generation to use for $\keff$ estimation
						\State $\mu_0$: initial law of particles
						
						\State\textbf{Initialise particle population:} Sample $\widetilde{X}_j^0 \sim \mu_0$ and	$X_j^0 \sim M(\widetilde{X}_j^0, \cdot )$ independently for $1 \le j \leq N$
						\State\textbf{For }{$t =1$ to $T+k$} \textbf{do}
						\State$\qquad$\textbf{If: }{$\mathcal{N}^{t-1} = \emptyset$}{$\quad$Error: ``System died out"}
						\State$\qquad$\textbf{EndIf}
						\State$\qquad$\textbf{Resample:} Sample $\widetilde{X}_j^t$ uniformly from $\mathcal{N}^{t-1}$ independently for $1 \leq j \leq N$
						\State$\qquad$\textbf{Evolve system:} Sample $X_j^t \sim M(\widetilde{X}_j^t, \cdot )$ independently for $1 \le j \leq N$
						\State$\qquad$\textbf{Generate new particle bank:} Add $n_j^t \in \N$ copies of $X_j^t$ to particle bank $\mathcal{N}^t$, with $\E[n_j^t\vert X_j^0] = G(X_j^t)$ for $1 \leq j \leq N$ 
						\State \textbf{EndFor}
					\end{algorithmic}
				\end{algorithm}

				\begin{remark}
					Note that Algorithm \ref{alg:genericPI} is indeed a version of the simple generic `resampling and evolution' procedure described in Section \ref{subsec:SMC4NTE}. 
								\end{remark}
				
				\begin{remark}
					In the nuclear engineering sometimes a slight variation of Algorithm \ref{alg:genericPI} is performed. Namely, instead of copying $X_j^t$, i.e. the particle that is about to undergo fission, one adds already the offspring (after fission) to $\mathcal{N}^t$. However, this is essentially still Algorithm~\ref{alg:genericPI} up to one extra iteration. 
				\end{remark}
				
				Next, we proceed to the theoretical analysis of Algorithm \ref{alg:genericPI}.
				To this end, for the rest of the section, let $(X_1^t, \dots, X_N^t)_{t \in \N_0}$ be the particles corresponding to Algorithm \ref{alg:genericPI}. Additionally, in this section $\tau(N)$ denotes the corresponding survival time, which equivalently to \eqref{equ:definition_extintion_time}, can be written as 
				\[
				\tau(N) \le \inf\left\{ t \in \N \,\colon\, G(X_1^t)=\cdots =G(X_N^t) =0   \right\}.
				\]
				It follows that for $t < \tau(N)$, we will always have $\sum_{j=1}^N
				G(X_j^{t}) > 0$.
				As before, for $t \in \N$, we denote the particle approximation of $\eta_t$ by $\eta_t^N$. That is, 
				\[
				\eta_t^N (f) = \frac{1}{N}\sum_{j=1}^N f(X_j^t), 
				\]
				for any bounded and measurable $f \colon E \to \R$. The filtration generated by the particles is denoted by $(\mathcal{F}_t)_{ t \in \N_0}$.

				\begin{proposition}\label{prop:Assumption_E_2_for_PWR_IT}
					For any $p \in [1, \infty)$, there exists a constant $C_p \in (0, \infty)$ such that for any $t, N \in \N$ we have 
					\[
					\E\left[ \mathds{1}_{\{\tau(N)>t\}} \left\vert \eta_t^N (f) - \Phi(\eta_{t-1}^N)(f)  \right\vert^p \bigg \vert \mathcal{F}_{t-1}  \right]^{1/p} \leq \frac{C_p \norm{f}}{\sqrt{N}}, 
					\]
					for any measurable and bounded $f \colon E \to \R$. 
				\end{proposition}
				
				\begin{proof}
					Note that on the set $\{\tau(N)>t\}$ we have 
					\[
					\Phi(\eta_{t-1}^N) (f) = \frac{Q\eta_{t-1}^N}{\eta_{t-1}^N(G)} = \frac{\sum_{j=1}^N G(X_j^{t-1}) M[f](X_j^{t-1})}{\sum_{j=1}^N G(X_j^{t-1})}.
					\]
					Thus, by specification of Algorithm \ref{alg:genericPI} we have that, again on the set $\{\tau(N)>t\}$ and given $\mathcal{F}_{t-1}$, $\eta_{t}^N(f)$ is an unbiased estimator for $\Phi(\eta_{t-1}^N) (f)$.
					Since particles evolve independently, the result follows by the Marcinkiewicz-Zygmund inequality, compare e.g. \cite[Chapter 8]{gut2006probability}. 
				\end{proof}
				
				Note that Proposition \ref{prop:Assumption_E_2_for_PWR_IT} verifies
				that Assumption \ref{assumption_E_p} is true, for any $p\in [1, \infty)$, for
				Algorithm~\ref{alg:genericPI}. Hence, by using what was proven in
				Section \ref{sec:L_p_error_bounds}, we immediately obtain the
				following two results. 
				
				\begin{corollary}\label{cor:error_bound_explicit_fission}
					For any $p \in [1, \infty)$ there exists a constant $C_p \in (0, \infty)$ such that for any $T,k,N \in \N$ and $\hat{\kfrak} = \kfrak_{T,k}^{ari}$ or $\hat{\kfrak}=\kfrak_{T,k}^{ari}$ we have
					\[
					\E \left[\mathds{1}_{\{\tau(N) > T+k\}} \left\vert \kfrak_{T,k}^{ari} - \keff  \right\vert^p   \right]^{1/p} \leq C_p\left( \frac{1}{\sqrt{N}} + \frac{\tau^T }{k} \right) , 
					\]
					and
					\[
					\E \left[ \left\vert \hat{\kfrak} - \keff  \right\vert^p   \right]^{1/p} \leq C_p \left( \frac{1}{\sqrt{N}} + \frac{\tau^T }{k} +  (T+k+1)^{1/2}e^{-(a N)/2}\right) , 
					\] 
					with $a = - \log(1-\kappa) \in (0, \infty)$, where $\kappa$ is as in Assumption \ref{ass:minimal_survival_probability}.
				\end{corollary}

				Again, with essentially identical arguments and
				conclusions, the above result can be modified to give a version of
				Theorem~\ref{thm_conditional_convergence} and
				Remark~\ref{rk-cost-calculations} in this context.

				\subsection{Variance reduction techniques}\label{subsec:variance_reduction}
				
				In this section we analyse possible modifications of Algorithm
				\ref{alg:genericPI}, which in the nuclear engineering literature are broadly referred to as \textit{variance reduction}, compare e.g. to \cite{bonnet2024effect}. Typically, those changes concern the evolution and resampling steps and result in an unbiased version of Algorithm \ref{alg:genericPI}. In practice, this helps to better `spread out' the particles. Examples are:
				\begin{enumerate}
					\item[1)] \textbf{Implicit fission:} Instead of
					evolving $\widetilde{X}_j^t$ via $M$ to its next fission event, a
					whole path of possible fission and scattering events is
					simulated. After each event, the remaining path is weighted by the probability of survival, such that the whole path remains unbiased.  
					\item[2)] \textbf{Splitting techniques:} Particles my be `split' into two (suitably weighted) ones if the original is in a particularly fissile region. Splitting is a tool to have more possible fission sites, whence helps to stabilise the simulation. 
				\end{enumerate}
				Now, we turn to the error bounds for the variance reduction techniques. For the remainder of this section suppose that Assumptions \ref{ass_perron_frobenius_assumption} and \ref{ass:minimal_survival_probability} are true.
				Due to the possible changes in Algorithm \ref{alg:genericPI}, e.g. through implicit fission or spitting, we consider a more general sequence of particles. 
				Namely, given $\widetilde{X}_1^{t-1}, \dots , \widetilde{X}_N^{t-1}$, each particle may give rise to a whole family of new (weighted) particles (fission events). For $1 \leq j \le N$, denote by $(X_{j,i}^t)_{i=1}^{\ell_j^t}$ the offspring of $\widetilde{X}_j^{t-1}$, where $\ell_j^t\in \N$ is the number of possible fission events corresponding to $X_j^{t-1}$.
				Additionally, let $\theta_{j,i}^t \in [0, 1]$, be a (random) weight associated to $X_{j,i}^t$.
				We define 
				\[
				\Gamma_j^t (\cdot ) = \sum_{i = 1}^{\ell_j^t} \theta_{j,i}^t \delta_{X_{j,i}^t}(\cdot) \qquad \text{and} \qquad \eta^N_t = \frac{1}{N}\sum_{j=1}^N \Gamma_j^t(\cdot ).
				\]
				For the rest of the section we assume that the weights are chosen such that each $\Gamma_j^t$, and hence also $\eta_{t}^N$, are probability measures.
				The objects $\tau(N)$; $(\mathcal{F}_t)_{t \in \N_0}$; $\kfrak_{T, k}^{ari}$ and $\kfrak_{T,k}^{geo}$ are defined similarly as in the previous section.  
				
				Investigating the proof of Proposition \ref{prop:Assumption_E_2_for_PWR_IT}, we immediately obtain the following. 
				\begin{lemma}\label{lem:Assumption_E_2_for_variance_reduction}
					Assume that for any $t, N \in \N$ and any bounded and measurable $f \colon E \to \R$, given $\mathcal{F}_{t-1}$, the random variables $\Gamma_1^{t}(f), \dots, \Gamma_N^t(f)$ are independent unbiased estimators of $\Phi\left(\delta_{\tilde{X}_j^{t-1}}\right)\left(f\right)$.
					Then, for any $p \in [1, \infty)$, there exists $C_p \in (0, \infty)$ such that for any $t, N \in\N$ we have
					\[
					\E\left[ \mathds{1}_{\{\tau(N)> t\}} \left\vert \eta_{t}^N(f) - \Phi(\eta_{t-1}^N)(f)  \right\vert^p  \right]^{1/p} \leq \frac{C_p \norm{f}}{\sqrt{N}},
					\] 
					for any measurable and bounded $f \colon E \to \R$.
				\end{lemma}
				
				Now, we observe that all the proofs of Section \ref{sec:L_p_error_bounds} are still true for the particle approximation $\eta_{t}^N$ as defined in this section. Thus, we obtain the following error bounds for the variance reduction techniques of this section.

				\begin{corollary}\label{cor:error_bound_variance_reduction}
					Assume that for any $t, N \in \N$ and any bounded and measurable $f \colon E \to \R$, given $\mathcal{F}_{t-1}$, the random variables $\Gamma_1^{t}(f), \dots, \Gamma_N^t(f)$ are independent unbiased estimators of $\Phi(\eta_{t-1}^N)(f)$.
					Then, for any $p \in [1, \infty)$ there exists $C_p \in (0, \infty)$ such that for any $T, N,k \in\N$ and $\hat{\kfrak} = \kfrak_{T,k}^{ari}$ or $\hat{\kfrak}=\kfrak_{T,k}^{ari}$ we have
					\[
					\E \left[\mathds{1}_{\{\tau(N) > T+k\}} \left\vert \hat{\kfrak} - \keff  \right\vert^p   \right]^{1/p} \leq C_p \left( \frac{1}{\sqrt{N}} + \frac{\tau^T }{k} \right),
					\] 
					and
					\[
					\E \left[ \left\vert \hat{\kfrak} - \keff  \right\vert^p   \right]^{1/p} \leq C_p\left( \frac{1}{\sqrt{N}} + \frac{\tau^T }{k} +  (T+k+1)^{1/2}e^{-(a N)/2}\right),
					\]
					with $a = - \log(1-\kappa) \in (0, \infty)$, where $\kappa$ is as in Assumption \ref{ass:minimal_survival_probability}. 
				\end{corollary}

				\begin{remark}
					It is possible to consider even more general version for $\Gamma_{j}^t$ in Corollaries \ref{cor:error_bound_variance_reduction} as long as the independence and unbiasedness properties are satisfied. 
				\end{remark}

				\section*{Acknowledgments}
				This work was supported by the EPSRC grants
				EP/W026899/2 and EP/P009220/1. TD was supported by a scholarship from the EPSRC Centre for Doctoral
				Training in Statistical Applied Mathematics at Bath (SAMBa), under the project EP/L015684/1.
				
				{Some of the results in this paper originally
					appeared in similar forms in the PhD thesis of TD.}
				
				The authors thank Th{\'e}ophile Bonnet for useful discussion and explanation about implementations of Monte Carlo power iteration aspects.  
				
				The authors used GenAI to assist them searching for relevant papers and for verification of some calculations. All GenAI suggestions were carefully checked for their correctness. All writing and ideas are original and were not produced by AI.

\end{document}